\documentclass[a4paper,reqno,11pt]{amsart}

\usepackage{amssymb}
\usepackage{amstext}
\usepackage{amsmath}
\usepackage{amscd}
\usepackage{amsthm}
\usepackage{amsfonts}
\usepackage{enumerate}
\usepackage{graphicx}
\usepackage{color}
\usepackage{here} 
\usepackage{bm} 
\usepackage{mathrsfs}
\usepackage{mathtools}
\usepackage{latexsym}
\usepackage{booktabs}
\usepackage{fancyhdr}
\usepackage{subcaption}
\usepackage{tikz}
\usepackage{tikz-cd}
\usetikzlibrary{arrows}
\usetikzlibrary{decorations.markings}
\usetikzlibrary{patterns,decorations.pathreplacing}

\usepackage{hyperref}
\hypersetup{
  colorlinks=true,
  citecolor=red,
  linkcolor=blue
}

\makeatletter
  
\@addtoreset{equation}{section}
\@namedef{subjclassname@2020}{\textup{2020} Mathematics Subject Classification}
\makeatother

\newcommand{\calA}{\mathcal{A}}

\newcommand{\calC}{\mathcal{C}}
\newcommand{\calD}{\mathcal{D}}

\newcommand{\calH}{\mathcal{H}}

\newcommand{\calL}{\mathcal{L}}

\newcommand{\calO}{\mathcal{O}}

\newcommand{\calQ}{\mathcal{Q}}

\newcommand{\calS}{\mathcal{S}}
\newcommand{\calW}{\mathcal{W}}
\newcommand{\calX}{\mathcal{X}}

\newcommand{\ZZ}{\mathbb{Z}}

\newcommand{\RR}{\mathbb{R}}

\newcommand{\TT}{\mathbb{T}}
\newcommand{\kk}{\Bbbk}

\newcommand{\rmV}{\mathrm{V}}

\newcommand{\bfa}{\mathbf{a}}

\newcommand{\eb}{\mathbf{e}}

\newcommand{\nb}{\mathbf{n}}

\newcommand{\bfu}{\mathbf{u}}
\newcommand{\vb}{\mathbf{v}}

\newcommand{\bfx}{\mathbf{x}}

\newcommand{\fkp}{\mathfrak{p}}

\newcommand{\fkC}{\mathfrak{C}}

\newcommand{\fkX}{\mathfrak{X}}

\newcommand{\Hom}{\operatorname{Hom}}

\def\opn#1#2{\def#1{\operatorname{#2}}} 
\opn\Cl{Cl} \opn\conv{conv} \opn\deg{deg} \opn\rank{rank} \opn\Spec{Spec} \opn\Stab{Stab} \opn\aff{aff} \opn\div{div} \opn\GL{GL}
\opn\cone{cone} \opn\End{End} \opn\Hom{Hom} \opn\mod{mod} \opn\gldim{gldim} \opn\pdim{pdim} \opn\diag{diag} \opn\vert{vert}
\opn\Block{Block} \opn\Pyr{Pyr} \opn\max{max} \opn\min{min} \opn\ini{in} \opn\rev{rev} \opn\ker{ker} \opn\lat{lat} \opn\pull{pull} \opn\rev{rev} \opn\Sym{Sym} \opn\supp{supp} \opn\int{int} \opn\star{star} \opn\sign{sign}

\newtheorem{thm}{Theorem}[section]

\newtheorem{lemma}[thm]{Lemma}
\newtheorem{lem}[thm]{Lemma}
\newtheorem{prop}[thm]{Proposition}

\newtheorem{conj}[thm]{Conjecture}

\theoremstyle{definition}
\newtheorem{definition}[thm]{Definition}
\newtheorem{defi}[thm]{Definition}
\newtheorem{ex}[thm]{Example}

\theoremstyle{remark}
\newtheorem{remark}[thm]{Remark}

\begin{document}

\title{Conic divisorial ideals of toric rings and applications to Hibi rings and stable set rings}
\author{Koji Matsushita}
\address{Graduate School of Mathematical Sciences, University of Tokyo, Komaba, Meguro-ku, Tokyo 153-8914, Japan}
\email{koji-matsushita@g.ecc.u-tokyo.ac.jp}


\subjclass[2020]{
Primary 13C14; 
Secondary 13F65, 
14M25, 
16S38, 
05C25, 
}
\keywords{Conic divisorial ideals, Toric rings, Quasi-symmetric representations, Weakly-symmetric representations, Non-commutative crepant resolutions, Hibi rings, Stable set rings}

\maketitle

\begin{abstract}
In this paper, we study conic divisorial ideals of toric rings.
We provide an idea to determine them and we give a description of the conic divisorial ideals of Hibi rings and stable set rings of perfect graphs by using this idea.
We also characterize when Hibi rings or stable set rings are quasi-symmetric or weakly-symmetric.
Moreover, by using the description of the conic divisorial ideals, we construct a non-commutative crepant resolution (NCCR) of a special family of stable set rings. 
\end{abstract}


\section{Introduction}

Throughout this paper, let $\kk$ be an algebraically closed field of characteristic $0$, for simplicity.

\subsection{Backgrounds}
Let $C\subset \RR^d$ be a $d$-dimensional rational polyhedral cone.
We define the toric ring $R=\kk[C\cap \ZZ^d]$ as
\[
R=\kk[C\cap\ZZ^d]=\kk[t_1^{\alpha_1}\cdots t_d^{\alpha_d} : (\alpha_1,\ldots,\alpha_d)\in C\cap\ZZ^d].
\]
Toric rings are normal affine semigroup rings (and hence Cohen-Macaulay rings), and are of particular interest in the area of combinatorial commutative algebra.
For example, the Ehrhart ring of a lattice polytope is a typical example of toric rings and its commutative ring-theoretical properties have been well investigated.

Recently, conic divisorial ideals, which are a certain class of divisorial ideals (rank one reflexive modules) defined on toric rings, and their applications are well studied (see, e.g., \cite{Bru, BG1}, and so on).
It is known that up to isomorphism the conic divisorial ideals of $R$ are exactly the direct summands of $R^{1/k}$ for $k \gg 0$ (\cite[Proposition 3.6]{BG1}, \cite[Proposition 3.2.3]{SmVdB}), where $R^{1/k}=\kk[C \cap (1/k\ZZ)^d]$ is regarded as an $R$-module.
Since $R^{1/k}$ is a maximal Cohen-Macaulay (MCM, for short) $R$-module, conic divisorial ideals of $R$ are also MCM $R$-modules.

Conic divisorial ideals play important roles in the theory of non-commutative algebraic geometry as well as commutative algebra.
In fact, the following theorem says that the endomorphism ring of the direct sum of ``all" conic divisorial ideals of $R$ is an NCR of $R$:

\begin{thm}[{\cite[Corollary 6.2]{FMS}, \cite[Proposition 1.8]{SpVdB}}]\label{thm2}
For $k \gg 0$, the global dimension of $E:=\End_R(R^{1/k})$ is equal to $d$, and hence $E$ is a non-commutative resolution (NCR) of $R$.
\end{thm}

Moreover, the endomorphism ring of the direct sum of ``some" conic modules of $R$ may be a non-commutative crepant resolution (NCCR), which is a special class of NCRs and was introduced by Van den Bergh (\cite{VdB3}) in the context of non-commutative algebraic geometry.
In particular, NCCRs constructed in this way are called toric NCCRs.
Although toric rings always have NCRs as mentioned in Theorem~\ref{thm2}, there exists an example of a toric ring which has no toric NCCRs (\cite[Example 9.1]{SpVdB}).
Nevertheless, various classes of toric rings admitting toric NCCRs have also been extensively studied.
For instance, it is known that the following toric rings have toric NCCRs:
\begin{itemize}
\item Gorenstein toric rings whose class groups are $\ZZ$ (\cite{VdB3}); 
\item Gorenstein Hibi rings whose class groups are $\ZZ^2$ (\cite{Nak2}); 
\item $3$-dimensional Gorenstein toric rings (\cite{Bro, IU, SpVdB2}); 
\item Segre products of polynomial rings (\cite{HN});
\item Gorenstein edge rings of complete multipartite graphs (\cite{HM}).
\end{itemize} 
However, in general, the existence of NCCRs for Gorenstein toric rings is still open.

\medskip

In the construction of NCCRs, it is natural and important to classify MCM divisorial ideals (including conic ones) of certain class of toric rings. 
It has been investigated in some classes of toric rings.
For example, a classification of MCM divisorial ideals is given in the case of toric rings whose divisor class groups are $\ZZ$ or $\ZZ^2$ (\cite{Sta1,VdB1}).
Furthermore, a description of conic divisorial ideals is also given in the case of Hibi rings (\cite{HN}) and edge rings of complete multipartite graphs (\cite{HM}).

On the other hand, any toric ring whose divisor class group is isomorphic to $\ZZ^r$ can be described as the ring of invariants under an action of the torus $T\cong (\kk^{\times})^r$ on the symmetric algebra $\Sym W$ of a $T$-representation $W$ (see, e.g., \cite[Theorem 2.1]{Bru}).
Suppose that $W$ is generic and unimodular. Then the following facts are known:
\begin{itemize}
\item If $W$ is quasi-symmetric, then the toric ring $\Sym (W)^T$ has a toric NCCR (\cite[Theorem 1.19]{SpVdB}).
\item Suppose that $W$ is weakly-symmetric. Then, a description of the MCM divisorial ideals of the toric ring $\Sym (W)^T$ is given (\cite[Corollary 3.7]{SpVdB4}).
Moreover, all toric NCCRs of $\Sym (W)^T$ are derived equivalent (\cite[Theorem 1.3]{SpVdB4}).
\end{itemize}
Quasi-symmetric (resp. weakly-symmetric) representations are defined by using its weights (see Definition~\ref{def_qua}), and other properties associated with them are studied in \cite{SpVdB} (resp. \cite{SpVdB4}).
Since we want to determine MCM divisorial ideals and construct NCCRs, we are also interested in when toric rings are quasi-symmetric or weakly-symmetric.

\medskip

In light of the above observations and current situation, it is important to address the following three problems:
\begin{itemize}
    \item Describe the conic divisorial ideals of given toric rings.
    \item Characterize the conditions under which $W$ is quasi-symmetric or weakly-symmetric for a given toric ring $R\cong \Sym (W)^T$.
    \item Give a new family of toric rings which have toric NCCRs.
\end{itemize}
In this paper, we study the above problems for two classes of toric rings, \textit{Hibi rings} and \textit{stable set rings}.

\medskip

\subsection{Hibi rings}
Hibi introduced a class of normal Cohen-Macaulay domains $\kk[P]$ arising from posets $P$ (\cite{Hibi}),
which is a quotient of a polynomial ring.
On the other hand, in this paper, our description of toric ring $\kk[P]$ (see Section~\ref{subsec_Hibi}) is based on the order polytope of $P$ introduced by Stanley (\cite{Sta2}), and it seems different from the original one.
However, it is known that they are isomorphic.
Nowadays, the toric rings $\kk[P]$ are called Hibi rings of posets $P$.

The Hibi ring of a poset $P$ is a typical example of an algebra with straightening laws domain on $P$ and is of interest in the area of combinatorial commutative algebra. 
In fact, algebraic properties of Hibi rings have been well studied as mentioned in the previous subsection. 
Moreover, the relationships between Hibi rings and other toric rings, such as stable set rings and edge rings, were investigated when their divisor class groups have small rank (\cite{HM3}).

\medskip

\subsection{Stable set rings}
The terminology ``stable set ring'' is used in \cite{HS}, so we also use it. 
The stable set ring $\kk[\Stab_G]$ of a finite simple graph $G$ is a $\kk$-algebra defined from the stable set polytope $\Stab_G$, and stable set polytopes were introduced by Chv\'{a}tal (\cite{Ch}).
It is known that the stable set ring $\kk[\Stab_G]$ of a perfect graph $G$ coincides with the Ehrhart ring of the stable set polytope, and hence $\kk[\Stab_G]$ can be regarded as the toric ring arising from a rational polyhedral cone.
In addition, the facets of stable set polytopes are completely characterized in the case of perfect graphs (\cite{Ch}).
Thus, stable set rings and stable set polytopes behave well for perfect graphs.
Moreover, stable set polytopes of perfect graphs include a remarkable class of another kind of polytopes arising from posets, which are called chain polytopes and were also introduced by Stanley (\cite{Sta2}).

Recently, algebraic properties of stable set rings have been well studied as well as Hibi rings. 
For example, in \cite{HM3}, divisor class groups of stable set rings of perfect graphs $G$ are completely characterized in terms of $G$.
Furthermore, perfect graphs $G$ whose stable set rings have the class groups $\ZZ$ or $\ZZ^2$ are also characterized, and in this case, each stable set ring of $G$ is isomorphic to a certain Hibi ring.

\medskip

\subsection{Main Results}\label{sec_main}
First, we give a description of the conic divisorial ideals of our toric rings.
We present an approach to determine a region representing conic classes in the divisor class group of a toric ring (Lemma~\ref{conic}).
By using this lemma, we get the following theorems:
\begin{thm}[{see Theorem~\ref{thm:conic_hibi}}] 
\label{conic_thm_intro}
Let $\kk[P]$ be the Hibi ring of a poset $P$ that the Hasse diagram of $\widehat{P}$ has $d+1$ vertices and $n$ edges. 
Then, the conic divisorial ideals of $\kk[P]$ one-to-one correspond to the points in $\calC(P)\cap\ZZ^{n-d}$
$($see $(\ref{ccp})$ for the precise definition of $\calC(P)$$)$. 
\end{thm}
\noindent Actually, this result has already been given in \cite[Theorem 2.4]{HN}. However, its proof is insufficient (see Remark~\ref{insu} for details). 

\begin{thm}[{see Theorem~\ref{conic_stab}}] 
\label{conic_stab_intro}
Let $\kk[\Stab_G]$ be the stable set ring of a perfect graph $G$ with $n+1$ maximal cliques. 
Then, the conic divisorial ideals of $\kk[\Stab_G]$ one-to-one correspond to the points in $\calC(G)\cap\ZZ^n$ $($see $(\ref{conic_stab})$ for the precise definition of $\calC(G)$$)$.
\end{thm}

The second main result is to determine when Hibi rings and stable set rings of perfect graphs are quasi-symmetric or weakly-symmetric:

\begin{thm}\label{main2_hibi}
Let $\kk[P]$ be the Hibi ring of a poset $P$. We consider the following conditions:
\begin{itemize}
\item[(i)] $P$ is a poset whose Hasse diagram is depicted in Figure~\ref{poset};
\item[(ii)] $\kk[P]$ is isomorphic to the tensor product of a polynomial ring and some Segre products of two polynomial rings;
\item[(iii)] $\kk[P]$ is weakly-symmetric;
\item[(iv)] $\kk[P]$ is quasi-symmetric.
\end{itemize}
Then, (i), (ii) and (iii) are equivalent. Furthermore, if $\kk[P]$ is Gorenstein, then the above four conditions are equivalent.
\end{thm}

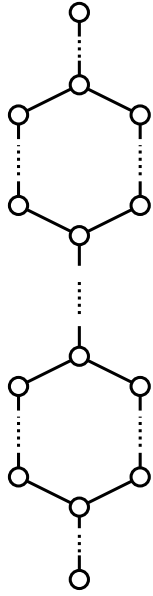
\begin{figure}[h]
\centering
{\scalebox{0.8}{
\begin{tikzpicture}[line width=0.05cm]

\coordinate (L1) at (0,2); \coordinate (L2) at (0,3.5); \coordinate (L3) at (0,6.5); 
\coordinate (L4) at (0,8.0); \coordinate (L5) at (0,6); \coordinate (L6) at (0,7);
\coordinate (R1) at (2,2); \coordinate (R2) at (2,3.5); \coordinate (R3) at (2,6.5); 
\coordinate (R4) at (2,8.0); \coordinate (R5) at (2,6); \coordinate (R6) at (2,7); 
\coordinate (M1) at (1,0.3); \coordinate (M2) at (1,1.5); \coordinate (M3) at (1,4.0); 
\coordinate (M4) at (1,6.0); \coordinate (M5) at (1,8.5); \coordinate (M6) at (1,9.7);
\coordinate (M7) at (1,8.2); \coordinate (M8) at (1,9); 

\draw (M1)--(1,0.7); \draw[dotted] (1,0.75)--(1,1.1); \draw (1,1.15)--(M2);
\draw (M3)--(1,4.5); \draw[dotted] (1,4.7)--(1,5.3); \draw (1,5.5)--(M4);
\draw (L1)--(0,2.4); \draw[dotted] (0,2.5)--(0,3.0); \draw (0,3.1)--(L2);
\draw (R1)--(2,2.4); \draw[dotted] (2,2.5)--(2,3.0); \draw (2,3.1)--(R2);
\draw (L3)--(0,6.9); \draw[dotted] (0,7.0)--(0,7.5); \draw (0,7.6)--(L4);
\draw (R3)--(2,6.9); \draw[dotted] (2,7.0)--(2,7.5); \draw (2,7.6)--(R4);
\draw (M5)--(1,8.9); \draw[dotted] (1,8.95)--(1,9.25); \draw (1,9.3)--(M6);
\draw (M2)--(L1); \draw (M2)--(R1); \draw (M3)--(L2); \draw (M3)--(R2);
\draw (M4)--(L3); \draw (M4)--(R3); \draw (M5)--(L4); \draw (M5)--(R4);

\draw [line width=0.05cm, fill=white] (L1) circle [radius=0.15];
\draw [line width=0.05cm, fill=white] (L2) circle [radius=0.15];
\draw [line width=0.05cm, fill=white] (L3) circle [radius=0.15];
\draw [line width=0.05cm, fill=white] (L4) circle [radius=0.15];
\draw [line width=0.05cm, fill=white] (R1) circle [radius=0.15]; 
\draw [line width=0.05cm, fill=white] (R2) circle [radius=0.15]; 
\draw [line width=0.05cm, fill=white] (R3) circle [radius=0.15]; 
\draw [line width=0.05cm, fill=white] (R4) circle [radius=0.15];
\draw [line width=0.05cm, fill=white] (M1) circle [radius=0.15]; 
\draw [line width=0.05cm, fill=white] (M2) circle [radius=0.15];
\draw [line width=0.05cm, fill=white] (M3) circle [radius=0.15];
\draw [line width=0.05cm, fill=white] (M4) circle [radius=0.15];
\draw [line width=0.05cm, fill=white] (M5) circle [radius=0.15]; 
\draw [line width=0.05cm, fill=white] (M6) circle [radius=0.15]; 

\end{tikzpicture} 
}}
\caption{The general X-shape poset}
\label{poset}

\end{figure}

\noindent Here, we call the poset in Figure~\ref{poset} \textit{general X-shape} in this paper, and define it in Section~\ref{sec_conic_stab} more precisely.

\begin{thm}\label{main2_stab}
Let $\kk[\Stab_G]$ be the stable set ring of a perfect graph $G$. 
We consider the following conditions:
\begin{itemize}
\item[(i)] $G$ has at most $2$ maximal cliques;
\item[(ii)] $\kk[\Stab_G]$ is isomorphic to the tensor product of a polynomial ring and the Segre products of two polynomial rings;
\item[(iii)] $\kk[\Stab_G]$ is weakly-symmetric;
\item[(iv)] $\kk[\Stab_G]$ is quasi-symmetric.
\end{itemize}
Then, (i), (ii) and (iii) are equivalent. Furthermore, if $\kk[\Stab_G]$ is Gorenstein, then the above four conditions are equivalent.
\end{thm}

Finally, we apply Theorem~\ref{conic_stab_intro} to construct NCCRs for a special family of stable set rings of perfect graphs.
We introduce a perfect graph $G_{r_1,\ldots,r_n}$ for $n\ge 3$ and $r_1,\ldots,r_n \in \ZZ_{>0}$ (see Section~\ref{family}),
and show that the stable set polytope $\Stab_{G_{r_1,\ldots,r_n}}$ does not coincide with any chain polytope (see Proposition~\ref{G} (iii)). 

\begin{thm}\label{main3}
Let 
$$\calL=\{(z_1,\cdots,z_{n}) \in \ZZ^n : 0\le z_i \leq r_i \text{ for $i\in [n]$} \}\subset \calC(G_{r_1,\ldots,r_n})\cap\ZZ^n.$$ 
Then, $E=\End_R(M_{\calL})$ is an NCCR of $R=\kk[\Stab_{G_{r_1,\ldots,r_n}}]$.
In particular, $\kk[\Stab_{G_{r_1,\ldots,r_n}}]$ has a toric NCCR.
\end{thm}

\medskip

\subsection{Organization}
In Section~\ref{sec_pre}, we recall some fundamental materials, e.g., toric rings, conic divisorial ideals, Hibi rings and stable set rings. 
We also prepare some notions and notation from graph theory in order to state our main results.
In Section~\ref{sec_conic}, we provide an idea to determine the conic divisorial ideals of toric rings and give a description of the conic divisorial ideals of Hibi rings and stable set rings of perfect graphs, that is, we prove Theorem~\ref{conic_thm_intro} and Theorem~\ref{conic_stab_intro}.
In Section~\ref{sec_qw}, we prove Theorem~\ref{main2_hibi} and Theorem~\ref{main2_stab}. 
In Section~\ref{sec_nccr}, we introduce a perfect graph $G_{r_1,\ldots,r_n}$ and discuss its properties.
We also recall the notion of NCCRs, prepare some lemmas and prove Theorem~\ref{main3}.  

\medskip

\subsection*{Acknowledgment} 
The author would like to thank Akihiro Higashitani for a lot of his helpful comments and instructive discussions. 
The author was partially supported by Grant-in-Aid for JSPS Fellows Grant JP22J20033. 

\bigskip

\section{Preliminaries}\label{sec_pre}
The goal of this section is to prepare the required materials for the discussions of our main results.


\subsection{Preliminaries on toric rings and conic divisorial ideals}\label{pre_toric}
In this paper, we mainly study two classes of toric rings, Hibi rings and stable set rings. Thus, we start this paper with introducing toric rings. 

Let $\langle -, -\rangle$ denote the natural inner product of $\RR^d$ and let $[n]=\{1,\ldots,n\}$ for $n\in \ZZ_{>0}$. 
We consider a $d$-dimensional strongly convex rational polyhedral cone 
$$
\tau=\mathrm{Cone}(v_1, \cdots, v_n)=\RR_{\ge 0}v_1+\cdots +\RR_{\ge 0}v_n 
$$
generated by $v_1, \cdots, v_n\in\ZZ^d$, where $d\le n$. 
We assume this system of generators is minimal and the generators are primitive, i.e., $\epsilon v_i \notin \ZZ^d$ for any $0<\epsilon <1$. 
For each generator, we define a linear form $\sigma_i(-)=\langle-, v_i\rangle$ and denote $\sigma(-)=(\sigma_1(-),\cdots,\sigma_n(-))$. 
We also consider the dual cone $\tau^\vee$ of $\tau$: 
$$
\tau^\vee=\{{\bf x}\in\RR^d : \sigma_i({\bf x})\ge0 \text{ for all } i\in [n] \}. 
$$
We now define the toric ring 
\begin{align}\label{R}
R=\kk[\tau^\vee\cap\ZZ^d]=\kk[t_1^{\alpha_1}\cdots t_d^{\alpha_d} : (\alpha_1, \cdots, \alpha_d)\in\tau^\vee\cap\ZZ^d]. 
\end{align}
Note that $R$ is a $d$-dimensional Cohen-Macaulay normal domain. 
In addition, for each $\bfa=(a_1, \cdots, a_n)\in\RR^n$, we set 
$$
\TT(\bfa)=\{{\bf x}\in\ZZ^d : \sigma_i({\bf x})\ge a_i \text{ for all } i\in [n] \}. 
$$
Then, we define the module $T(\bfa)$ generated by all monomials whose exponent vector is in $\mathbb{T}(\bfa)$. 
By the definition, we have $\mathbb{T}(0)=\tau^\vee\cap\ZZ^d$ and $T(0)=R$.
Moreover, we note some facts associated with the module $T(\bfa)$ (see, e.g., \cite[Section 4.F]{BG2}): 
\begin{itemize}
\item Since $\sigma_i(\bfx)\in \ZZ$ for any $i\in [n]$ and any $\bfx \in \ZZ^d$, we can see that $T(\bfa)=T(\ulcorner \bfa\urcorner)$, where $\ulcorner \; \urcorner$ means the round up 
and $\ulcorner \bfa\urcorner=(\ulcorner a_1\urcorner, \cdots, \ulcorner a_n\urcorner)$. 
\item The module $T(\bfa)$ is a divisorial ideal and any divisorial ideal of $R$ takes this form. 
Therefore, we can identify each $\bfa\in\ZZ^n$ with the divisorial ideal $T(\bfa)$.
\item It is known that the isomorphic classes of divisorial ideals of $R$ one-to-one correspond to the elements of the divisor class group $\Cl(R)$ of $R$. 
We see that for $\bfa, \bfa^\prime\in\ZZ^n$, $T(\bfa)\cong T(\bfa^\prime)$ if and only if there exists ${\bf y}\in \ZZ^d$ such that $a_i=a_i^\prime+\sigma_i({\bf y})$ for all $i\in [n]$. 
Thus, we have $\Cl(R)\cong\ZZ^n/\sigma(\ZZ^d)$.
\end{itemize} 
Let $\fkp_i= T(\eb_i)$, where $\eb_i$ is the $i$-th basic vector in $\ZZ^n$, 
and consider the prime divisor $\calD_i= \rmV(\fkp_i)=\Spec R/\fkp_i$ on $\Spec R$. 
Then, the divisorial ideal $T(\bfa)=T(a_1,\cdots,a_n)$ corresponds to the Weil divisor $-(a_1\calD_1+\cdots +a_n\calD_n)$. 
By the fact that $\Cl(R) \cong \ZZ^n/\sigma(\ZZ^d)$, we obtain that 
\begin{equation}\label{relation_divisor}
\sigma_1(\eb_j) \calD_1+\cdots+\sigma_n(\eb_j) \calD_n=v_1^{(j)}\calD_1+\cdots+v_n^{(j)}\calD_n=0
\end{equation}
in $\Cl(R)$ for all $j\in [d]$, where for a vector $v\in\RR^d$, $v^{(j)}$ denotes the $j$-th coordinate of $v$.

\medskip

We are interested in a divisorial ideal called \textit{conic}. 
\begin{definition}[{see, e.g., \cite[Section~3]{BG1}}]
\label{def_conic}
We say that a divisorial ideal $T(\bfa)$ is \emph{conic} if there exists $\bfx\in\RR^d$ such that $\bfa=\ulcorner\sigma(\bfx)\urcorner$. 
\end{definition}

\medskip

If $\Cl(R)$ is torsionfree, that is, $\Cl(R)\cong \ZZ^r$ for some $r\in \ZZ_{\ge 0}$, then we can rewrite $R$ as the ring of invariants under the action of $G= \Hom(\Cl(R), \kk^{\times})\cong (\kk^{\times})^r$ on $S= \kk[x_1,\ldots,x_n]$ as follows (we use the terminology and notation as in \cite[Section~3]{HN}):

Let $X(G)$ be the character group of $G$. 
We can see that $X(G)\cong \Cl(R)$, and hence we can use the same symbol for both of a character and the corresponding weight. 
When we consider the prime divisor $\calD_i$ on $\Spec R$ as the element in $X(G)\cong \ZZ^r$ via the surjection in (\ref{relation_divisor}), we denote it by $\beta_i$. 

For a character $\chi \in X(G)$, we also denote by $V_\chi$ the irreducible
representation corresponding to $\chi$, and we let $W=\bigoplus_{i=1}^n V_{\beta_i}$. 
Then, the symmetric algebra $\Sym W$ of the $G$-representation $W$ is isomorphic to $S$, and the algebraic torus $G$ acts on $S$, which is the action induced by $g\cdot x_i=\beta_i(g)x_i$ for $g\in G$. 
This action gives the $\Cl(R)$-grading on $S$, and the degree zero part coincides with the $G$-invariant components.
In particular, we have $R=S^G$ (see, e.g., \cite[Theorem 2.1]{BG1}). 

For a character $\chi$, we define $M_{\chi}= (S\otimes_{\kk} V_{\chi})^G$.
This is an $R$-module called the \textit{module of covariants} associated to $V_\chi$ and is generated by $f\in S$ with $g\cdot f=\chi(g)f$ for any $g\in G$. 
Note that for $\chi=\sum_i a_i\beta_i \in X(G)$, we have $T(a_1,\ldots,a_n)=M_{-\chi}$.

Moreover, by the arguments in \cite[Section 10.6]{SpVdB}, we can describe the conic divisorial ideals of $R$ by using the weights $\beta_1,\ldots,\beta_n$ as follows: 
\begin{prop}\label{strongly}
Let $\chi\in X(G)$. Then, $M_{-\chi}$ is conic if and only if 
one can write $-\chi=\sum_i a_i\beta_i$ with $a_i \in [0,1)$ for all $i\in [n]$.
\end{prop}


We introduce notions of quasi-symmetric and weakly-symmetric:

\begin{definition}[{\cite[Definition 2.2]{SpVdB4}}]\label{def_qua}
A $G$-representation $W$ is \textit{quasi-symmetric} if for every line $l\subset X(G)_{\RR}= X(G)\otimes_{\ZZ}\RR$ passing through the origin, we have $\sum_{\beta_i \in l}\beta_i=0$. 
It is \textit{weakly-symmetric} if for every line $l$, the cone spanned by $\beta_i\in l$ is either zero or $l$. 
We say that a toric ring $R$ is quasi-symmetric (resp. weakly-symmetric) if $R\cong S^G$ with $S=\Sym W$ and $W$ is a quasi-symmetric (resp. weakly-symmetric) representation.
\end{definition}
\noindent Note that quasi-symmetric representations are weakly-symmetric. 
If $W$ is quasi-symmetric, then the top exterior $W$ is the trivial representation, and hence $R=S^G$ is Gorenstein.

\medskip

\subsection{Preliminaries on Hibi rings}
\label{subsec_Hibi}
In this subsection, we recall the order polytopes and Hibi rings of posets. 

Let $P=\{p_1,\cdots,p_{d-1}\}$ be a finite partially ordered set (poset, for short) equipped with a partial order $\preceq$. 
For a subset $I \subset P$, we say that $I$ is a \textit{poset ideal} of $P$ if $p \in I$ and $q \preceq p$ imply $q \in I$. 
For a subset $A \subset P$, we call $A$ an \textit{antichain} of $P$ if $p \not\preceq q$ and $q \not\preceq p$ for any $p,q \in A$ with $p \neq q$. 
Note that $\emptyset$ is regarded as a poset ideal and an antichain. 
Let 
\begin{align*}
\calO_P=\{(x_1,\cdots,x_{d-1}) \in \RR^{d-1} : x_i \geq x_j \text{ if }p_i \preceq p_j \text{ in }P, \; 0 \leq x_i \leq 1 \text{ for } i\in[d-1]\}.
\end{align*}
This convex polytope $\calO_P$ is called the \textit{order polytope} of a poset $P$. 
According to \cite{Sta2}, it is known that the elements of $\calO_P\cap \ZZ^{d-1}$ are precisely the vertices of $\calO_P$ and they are $(0,1)$-vectors, that is, $\calO_P$ is a $(0,1)$-polytope.
Furthermore, the vertices one-to-one correspond to the poset ideals of $P$.
In fact, a $(0,1)$-vector $(a_1,\ldots,a_{d-1})$ is a vertex of $\calO_P$ if and only if $\{ p_i \in P : a_i =1 \}$ is a poset ideal.

For a poset $P$, let $\kk[P]$ denote a toric ring defined by setting 
$$\kk[P]=\kk[{\bf t}^\alpha t_0 : \alpha \in \calO_P \cap \ZZ^{d-1}]=\kk\left[\; \left(\prod_{p_i\in I}t_i\right) t_0 : \text{$I$ is a poset ideal of $P$}\right],$$
where ${\bf t}^\alpha=t_1^{\alpha_1}\cdots t_{d-1}^{\alpha_{d-1}}$ for $\alpha=(\alpha_1,\cdots,\alpha_{d-1}) \in \ZZ^{d-1}$. 
This $\kk$-algebra is called the {\em Hibi ring} of $P$. 
The following properties associated with order polytopes and Hibi rings of posets $P$ are known: 
\begin{itemize}
\item $\calO_P$ has the integer decomposition property (IDP, for short), that is, for any $n\in \ZZ_{>0}$ and any $\alpha \in n\calO_P\cap \ZZ^{d-1}$, there exist $\alpha_1,\ldots,\alpha_n \in \calO_P\cap \ZZ^{d-1}$ such that $\alpha=\alpha_1+\cdots +\alpha_n$, and hence $\kk[P]$ coincides with the Ehrhart ring of $\calO_P$ (see, e.g., \cite[Section 10.4]{Villa} for Ehrhart rings).
\item $\kk[P]$ is Gorenstein if and only if $P$ is pure (\cite{Hibi}), 
where we say that $P$ is \emph{pure} if all the maximal chains $p_{i_1} \prec \cdots \prec p_{i_\ell}$ have the same length. 
\end{itemize}

\medskip

The Hibi ring of a poset can be described as the toric ring arising from a rational polyhedral cone as follows.
Let $P=\{p_1,\cdots,p_{d-1}\}$. For $p_i, p_j \in P$ with $p_j \prec p_i$, we say that $p_i$ {\em covers} $p_j$ 
if there is $p \in P$ with $p_j \preceq p \preceq p_i$ then $p=p_j$ or $p=p_i$.
Thus, the edge $\{p_i,p_j\}$ of the Hasse diagram $\calH(P)$ of $P$ if and only if $p_i$ covers $p_j$ or $p_j$ covers $p_i$.
Set $\widehat{P}=P \cup \{\hat{0}, \hat{1}\}$, 
where $\hat{0}$ (resp. $\hat{1}$) is the unique minimal (resp. maximal) element not belonging to $P$. 
Let us denote $p_0=\hat{0}$ and $p_d=\hat{1}$.
For each edge $e=\{p_i,p_j\}$ of $\calH(\widehat{P})$ with $p_i \prec p_j$, 
let $\sigma_e$ be a linear form in $\RR^d$ defined by 
\begin{align*}
\sigma_e({\bf x})=
\begin{cases}
x_i-x_j \;&\text{ if }j \not= d, \\
x_i \; &\text{ if }j=d 
\end{cases}
\end{align*}
for ${\bf x}=(x_0,x_1,\cdots,x_{d-1})$. Let $\tau_P=\mathrm{Cone}(\sigma_e : e \text{ is an edge of }\calH(\widehat{P}))\subset \RR^d$. 
Then, we can see that $\kk[P]=\kk[\tau_P^\vee \cap \ZZ^d]$.
Let $e_1,\cdots,e_n$ be all the edges of $\calH(\widehat{P})$. We set a linear form $\sigma_P:\RR^d \rightarrow \RR^n$ by 
$$\sigma_P({\bf x})=(\sigma_{e_1}({\bf x}),\cdots,\sigma_{e_n}({\bf x})) \in \RR^n$$ 
for ${\bf x} \in \RR^d$.

\medskip

Let $P$ and $Q$ be two posets with $P\cap Q=\emptyset$. 
The \textit{disjoint union} of $P$ and $Q$ is the poset $P+Q$ on $P\cup Q$ such that $x\preceq y$ in $P+Q$ if (a) $x,y\in P$ and $x\preceq y$ in $P$, or (b) $x,y\in Q$ and $x\preceq y$ in $Q$.
The \textit{ordinal sum} of $P$ and $Q$ is the poset $P\oplus Q$ on $P\cup Q$ such that $x \preceq y$ in $P\oplus Q$ if
(a) $x,y\in P$ and $x\preceq y$ in $P$, (b) $x,y\in Q$ and $x\preceq y$ in $Q$, or (c) $x\in P$ and $y\in Q$.
Moreover, let $P\oplus'Q=P\oplus \{z\}\oplus Q$, where $z$ is a new element which is not contained in $P\cup Q$.
Note that the equality $P\oplus Q=Q\oplus P$ does not hold in general while the equality $P+Q=Q+P$ holds.
By observing poset ideals of $P+Q$ and $P\oplus' Q$, the following proposition holds:

\begin{prop}\label{p_sum} Let $P$ and $Q$ be two posets with $P\cap Q=\emptyset$.
\begin{itemize}
\item[(i)] $\kk[P+Q]$ is isomorphic to the Segre product of $\kk[P]$ and $\kk[Q]$.
\item[(ii)] $\kk[P\oplus' Q]\cong \kk[Q\oplus' P] \cong \kk[P]\otimes_{\kk}\kk[Q]$. Moreover, $\kk[P\oplus' Q]$ is quasi-symmetric (resp. weakly-symmetric) if and only if $\kk[P]$ and $\kk[Q]$ are quasi-symmetric (resp. weakly-symmetric).
\end{itemize}
\end{prop}
\noindent Here, the \textit{Segre product} of two standard algebras $R=\bigoplus_{n\ge 0}R_n$ and $S=\bigoplus_{n\ge 0}S_n$ over $\kk$ is the graded subalgebra $\bigoplus_{n\ge 0}(R_n\otimes_{\kk}S_n)$ of $R\otimes_{\kk}S$.

\medskip

We also recall another polytope arising from $P$, which is defined as follows: 
\begin{align*}
\calC_P=\{(x_1,\ldots,x_d) \in \RR^d : \;&x_i \geq 0 \text{ for }i=1,\ldots,d, \\
&x_{i_1}+\cdots+x_{i_k} \leq 1 \text{ for }p_{i_1} \prec \cdots \prec p_{i_k} \text{ in } P\}.\end{align*} 
A convex polytope $\calC_P$ is called the \textit{chain polytope} of $P$. 
Similarly to order polytopes, it is known that $\calC_P$ is a $(0,1)$-polytope and the vertices of $\calC_P$ one-to-one correspond to the antichains of $P$ (\cite{Sta2}). 

\begin{thm}[{\cite[Theorem 2.1]{HL16}}]\label{X}
Let $P$ be a poset. Then $\calO_P$ and $\calC_P$ are unimodularly equivalent (i.e., there are a vector $v\in \ZZ^d$ and a unimodular transformation $f\in \GL_d(\ZZ)$ such that $\calC_P=f(\calO_P)+v$) if and only if $P$ does not contain the X-shape subposet. 
\end{thm}
\noindent Here, the \textit{X-shape} poset is a poset $\{p_1,p_2,p_3,p_4,p_5\}$ equipped with the partial orders $p_1 \prec p_3 \prec p_4$ and $p_2 \prec p_3 \prec p_5$.

\medskip

\subsection{Preliminaries on stable set rings}
In this subsection, we recall stable set polytopes and stable set rings of graphs. 
For the fundamental materials on graph theory, consult, e.g., \cite{Die}. 

For a simple graph $G$, let $V(G)=\{1,\ldots,d\}$ denote the vertex set of $G$ and let $E(G)$ denote the edge set of $G$.
We say that $S \subset V(G)$ is a \textit{stable set} or an \textit{independent set} (resp. a \textit{clique}) 
if $\{v,w\} \not\in E(G)$ (resp. $\{v,w\} \in E(G)$) for any distinct vertices $v,w \in S$. 
Note that the empty set and each singleton are regarded as stable sets. 

Given a subset $W \subset V(G)$, let $\rho(W)=\sum_{i\in W}\eb_i \in \RR^d$, where $\rho(\emptyset)$ stands for the origin of $\RR^d$. 
We define a lattice polytope associated with a graph $G$ as follows: 
\begin{align*}
\Stab_G=\conv(\{\rho(S) : S \text{ is a stable set}\}). 
\end{align*}
We call $\Stab_G$ the \textit{stable set polytope} of $G$. 
In addition, let $\kk[\Stab_G]$ denote a $\kk$-algebra defined by setting 
$$\kk[\Stab_G]=\kk[{\bf t}^\alpha t_0 : \alpha \in \Stab_G \cap \ZZ^d]=\kk\left[\; \left(\prod_{i\in S}t_i\right) t_0 : \text{$S$ is a stable set of $G$}\right],$$
which is called the {\em stable set ring} of $G$.

\medskip

In what follows, we treat the stable set rings of \textit{perfect graphs} since the following properties hold: 
\begin{itemize}
\item $\Stab_G$ is compressed if and only if $G$ is perfect (\cite{OH01,GPT}).
This implies that $\kk[\Stab_G]$ is normal if $G$ is perfect. 
Moreover, $\Stab_G$ has IDP if $G$ is perfect. 
Therefore, $\kk[\Stab_G]$ coincides with the Ehrhart ring of $\Stab_G$.
\item $\kk[\Stab_G]$ is Gorenstein if and only if $G$ all maximal cliques of $G$ have the same cardinality (\cite[Theorem 2.1]{OH06}).
\item The facets of $\Stab_G$ are completely characterized when $G$ is perfect (\cite[Theorem 3.1]{Ch}). 
From its description, we have 
\begin{align*}
\Stab_G=\{(x_1,\cdots,x_{d}) \in \RR^{d} \mid \, & x_i\geq 0 \text{ for }i\in [d], \;\\ & \, 1-\sum_{i\in Q}x_i\geq 0 \text{ for each maximal clique $Q$ of $G$}\}.
\end{align*}
\end{itemize}

\noindent From these properties, the stable set ring of a perfect graph can be described as the toric ring arising from a rational polyhedral cone as well as Hibi rings.
For a perfect graph $G$ with maximal cliques $Q_0,Q_1,\ldots,Q_n$ and $i\in \{0,1,\ldots,n+d\}$, 
let $\sigma_i$ be a linear form in $\RR^{d+1}$ defined by 
\begin{align*}
\sigma_i({\bf x})=
\begin{cases}
x_{0}-\sum_{j\in Q_i}x_j \;&\text{ if } i\in \{0,1,\ldots,n\}, \\
x_{i-n} \; &\text{ if }i\in \{n+1,\ldots,n+d\} 
\end{cases}
\end{align*}
for ${\bf x}=(x_0,x_1,\cdots,x_{d})$. Let $\tau_G=\mathrm{Cone}(\sigma_i : i\in \{0,1,\ldots,n+d\})\subset \RR^{d+1}$. 
Then, we can see that $\kk[\Stab_G]=\kk[\tau_G^\vee \cap \ZZ^d]$. 
We set a linear form $\sigma_G:\RR^{d+1} \rightarrow \RR^{n+d+1}$ by 
$$\sigma_G({\bf x})=(\sigma_0({\bf x}),\cdots,\sigma_{n+d}({\bf x})) \in \RR^{n+d+1}$$ 
for ${\bf x} \in \RR^{d+1}$.

\medskip

Given a poset $P$, we define the \textit{comparability graph} of $P$, denoted by $G(P)$, 
as a graph on the vertex set $V(G(P))=[d]$ with the edge set 
$$E(G(P))=\{\{i,j\} : \text{$p_i$ and $p_j$ are comparable in $P$}\}.$$ 
It is known that $G(P)$ is perfect for any $P$ (see, e.g., \cite[Section 5.5]{Die}) and the cliques of $G(P)$ one-to-one correspond to the chains of $P$. 
Moreover, we see that $\calC_P=\Stab_{G(P)}$ and if $P$ does not contain the X-shape, then the stable set ring 
$\kk[\calC_P]=\kk[\Stab_{G(P)}]$ is isomorphic to $\kk[P]$ by Theorem~\ref{X}.

\medskip

\subsection{Preliminaries on graph theory}
At the end of this section, we prepare some more notions and notation on (directed) graphs.
We refer the reader to e.g., \cite{Die, GR} for the introduction to graph theory.

Let $G$ be a graph.
For a subset $W \subset V(G)$, let $G_W$ denote the induced subgraph with respect to $W$.
For a vertex $v$, we denote by $G \setminus v$ instead of $G_{V(G) \setminus \{v\}}$. 
Similarly, for $S \subset V(G)$, we denote by $G \setminus S$ instead of $G_{V(G) \setminus S}$. 
We sometimes denote a cycle $C$ of $G$ by $(v_1,\ldots,v_m)$ (i.e., $V(C)=\{v_1,\ldots,v_m\}\subset V(G) $, $v_i\neq v_j$ for $i,j\in [m]$ and $E(C)=\{\{v_i,v_{i+1}\}\in E(G) : i\in [m]\}$, where $v_{m+1}=v_1$).
We say that a cycle $C=(v_1,\ldots,v_m)$ of $G$ is a \textit{circuit} if $C$ has no chords, that is, $E(G_{V(C)})=E(C)$. 
We also say that a graph is \textit{chordal} if each of its cycles of length at least $4$ has a chord. 
If $G$ is a graph with induced subgraphs $G_1, G_2$ and $S$, such that $V(G)=V(G_1)\cup V(G_2)$ and $V(S)=V(G_1) \cap V(G_2)$, we say that $G$ arises from $G_1$ and $G_2$ by \textit{pasting} these graphs together along $S$.
A characterization of chordal graphs is known as follows:
\begin{prop}[cf. {\cite[Proposition 5.5.1]{Die}}]\label{past}
A graph is chordal if and only if it can be constructed recursively by pasting along complete subgraphs, starting from complete graphs.
\end{prop}
\noindent Moreover, it is known that every chordal graph is perfect (cf. \cite[Proposition 5.5.2]{Die}).

\medskip

For a connected graph $G$, a subgraph $T$ of $G$ is called a \textit{spanning tree} if $T$ is a connected graph with $V(T)=V(G)$ and contains no cycles.
For each $e\in E(G)\setminus E(T)$, there is a unique cycle $C_e$ in $T+e$, where $T+e$ is the subgraph of $G$ on the vertex set $V(T)$ with the edge set $E(T)\cup \{e\}$. We call $C_e$ the \textit{fundamental cycle} of $e$ with respect to $T$.

The \textit{flow space} of a directed graph $A$ is the subspace of $\RR^{E(A)}$ generated by the vectors $x\in \RR^{E(A)}$ such that $D_A x=0$, where $D_A$ is the \textit{incidence matrix} of $A$, which is the $\{0,\pm1\}$-matrix with rows and columns indexed by the vertices and edges of $A$, respectively, such that the $ve$-entry of $D_A$ is equal to $1$ if the vertex $v$ is the head of the edge $e$, $-1$ if $v$ is the tail of $e$, and $0$ otherwise. 
Let $C=(v_1,\ldots,v_m)$ be a cycle in $A$. 
Using the orientation of $A$, the cycle $C$ determines an element $\vb(C) \in \RR^{E(A)}$ as follows: 
\begin{align*}
\vb(C)^{(e)}=
\begin{cases}
0 \; &\text{ if $e\notin E(C)$}, \\
1 \; &\text{ if $e=\{v_i,v_{i+1}\}$ and $v_{i+1}$ is the head of $e$}, \\
-1 \; &\text{ if $e=\{v_i,v_{i+1}\}$ and $v_{i+1}$ is the tail of $e$}.
\end{cases}
\end{align*}
We refer to $\vb(C)$ as the \textit{signed characteristic vector} of $C$. 
It is known that the signed characteristic vectors of the fundamental cycles with respect to a spanning tree of $A$ form bases of the flow space of $A$. 
For a cycle $C$ in $A$, we set
\begin{align*}
\supp^+(C)=\{e\in E(C) : \vb(C)^{(e)}>0\} \quad \text{ and } \quad \supp^-(C)=\{e\in E(C) : \vb(C)^{(e)}<0\}.
\end{align*}

\bigskip

\section{Conic divisorial ideals of toric rings}\label{sec_conic}

In this section, we discuss conic divisorial ideals of toric rings whose divisor class group is a free abelian group and give an idea to describe them. 
In particular, we determine conic divisorial ideals of Hibi rings and stable set rings by using the idea.

\subsection{Description of conic divisorial ideals of toric rings}
Throughout this subsection, let $R$ be the toric ring defined in (\ref{R}) and we assume that the divisor class group of $R$ is isomorphic to $\ZZ^r$.
Moreover, let $\beta_1,\ldots,\beta_n$ be the weights of $R$.
Furthermore, let $\bar{\beta}_1,\ldots,\bar{\beta}_{n'}$ be weights of $R$ such that $n'$ is the minimal number with $\{\bar{\beta}_1,\ldots,\bar{\beta}_{n'}\}=\{\beta_1,\ldots,\beta_n\}$, and let $m_i$ be the multiplicity of $\bar{\beta}_i$ for $i\in [n']$, that is, $m_i=|\{j\in [n] : \beta_j=\bar{\beta}_i\}|$.
By Proposition~\ref{strongly}, each element of $\calW(R) \cap \ZZ^r$ one-to-one corresponds to a conic divisorial ideal of $R$, where
\begin{align*}
\calW(R) = \Big\{ \sum_{i=1}^n a_i\beta_i \in \RR^r : a_i\in [0,1) \Big\}=\Big\{ \sum_{i=1}^{n'} \bar{a}_i\bar{\beta}_i \in \RR^r : \bar{a}_i\in [0,m_i) \Big\}.
\end{align*}
On the other hand, we define 
$$\calW'(R) = \Big\{\sum_{i=1}^n a_i\beta_i \in \RR^r : a_i\in [0,1] \Big\}=\Big\{ \sum_{i=1}^{n'} \bar{a}_i\bar{\beta}_i \in \RR^r : \bar{a}_i\in [0,m_i] \Big\}.$$
Note that $\calW'(R)$ is a lattice polytope since it is the Minkowski sum of lattice segments $\{\bar{a}_i\bar{\beta}_i : \bar{a}_i\in [0,m_i]\}$.
According to oriented matroid theory, we can determine the faces of $\calW'(R)$ as follows. 
We define the sign function $\sign : \RR \to \{+,-,0\}$ by setting
\begin{align*}
\sign(x)=\begin{cases}
+ \; &\text{ if } x>0, \\
0 \; &\text{ if } x=0, \\
- \; &\text{ if } x<0, \end{cases}
\end{align*}
and define partial order on $\{+,-,0\}$ by setting $0\prec +$ and $0\prec -$, while $+$ and $-$ are incomparable.
We consider the subset of $\{+,-,0\}^{n'}$:
$$S=\{(\sign(\langle \bar{\beta}_1, \nb \rangle), \ldots, \sign(\langle \bar{\beta}_{n'}, \nb \rangle)) : \nb \in \RR^{n'}\setminus \{0\}\}.$$
Note that $S$ can be regarded as a poset by using componentwise partial ordering: for $s,s'\in S$, $s\preceq s'$ if and only if $s^{(i)} \preceq s'^{(i)}$ for all $i\in [n']$.
By \cite[Proposition 2.2.2]{BLSWZ}, there is an order-reversing bijection between $S$ and the set of faces of $\calW'(R)$ (except for the empty set and $\calW'(R)$ itself), partially ordered by inclusion. 
In particular, by considering the correspondence between the facets of $\calW'(R)$ and the minimal elements of $S$, the following lemma holds:
\begin{lem}\label{lem}
If there exist $\nb\in \ZZ^r\setminus \{0\}$ and $\bar{\beta}_{i_1},\ldots,\bar{\beta}_{i_{r-1}}$ such that $\bar{\beta}_{i_1},\ldots,\bar{\beta}_{i_{r-1}}$ are linearly independent and $\langle \nb, \bar{\beta}_{i_j} \rangle=0$ for all $j\in [r-1]$, then 
$$F=\Big\{\sum_{\langle \nb, \beta_i \rangle> 0 } \beta_i+\sum_{\langle \nb, \beta_i \rangle= 0 } a_i\beta_i \in \RR^r : a_i\in [0,1] \Big\}$$
is a facet of $\calW'(R)$. Conversely, all facets of $\calW'(R)$ are obtained in this way.
\end{lem}

Our goal is to determine the facet defining inequalities of a convex polytope representing conic classes. 
Let $m\in \ZZ_{>0}$ and let $p_i, q_i\in \ZZ_{>0}$ for $i\in [m]$.
Moreover, for $i\in[m]$ and $j\in[r]$, let $c_{ij}$ be an integer such that the greatest common divisor of $c_{i1},\ldots,c_{ir}$ is equal to $1$ for all $i\in [m]$. 
We define two convex polytopes:
\begin{align*}
\calC&=\{(z_1,\ldots,z_r)\in \RR^r : -q_i \le \sum_{j=1}^r c_{ij}z_j \le p_i \text{ for all }i\in [m]\} \text{ and } \\
\calC'&=\{(z_1,\ldots,z_r)\in \RR^r : -q_i-1 \le \sum_{j=1}^r c_{ij}z_j \le p_i+1 \text{ for all }i\in [m]\}. 
\end{align*}
Note that if $\calC'$ is a lattice polytope, then we have $\int(\calC')\cap\ZZ^r=\calC\cap\ZZ^r$, where $\int(\calC')$ denotes the relative interior of $\calC'$.

\medskip

The following lemma is useful for describing conic divisorial ideals of toric rings.

\begin{lem}\label{conic}
\begin{itemize}
\item[(i)] If $\calW'(R)=\calC'$, then one has $\calW(R)\cap \ZZ^r=\calC\cap\ZZ^r$.
\item[(ii)] Suppose that $\calW'(R)\subset \calC'$. If all vertices of $\calC'$ are in $\calW'(R)$, then $\calW'(R)=\calC'$.
\end{itemize}
\end{lem}

\begin{proof}
(i) We show that $\int(\calW'(R))=\calW(R)$.
This implies 
\begin{align}\label{prop3.1}
\calW(R)\cap\ZZ^r=\int(\calW'(R))\cap\ZZ^r=\int(\calC')\cap\ZZ^r=\calC\cap\ZZ^r.
\end{align}
Note that $\dim \calW'(R)=\dim \calC'=r$.
For any $\beta \in \int(\calW'(R))$, there exists $k>1$ such that $k\beta \in \calW'(R)$. Thus, we have $\beta \in \calW(R)$ and hence $\int(\calW'(R))\subset \calW(R)$.

To prove the reverse inclusion, we need only show that if $\beta \in \calW'(R)$ is in the boundary $\partial\calW'(R)$ of $\calW'(R)$, then $\beta \notin \calW(R)$.
Let $\calS=\sigma(\tau^{\vee}\cap \ZZ^d)$ and let $\pi_i : \RR^n \to \RR$ be the $i$-th projection for each $i\in [n]$.
Note that the group of differences $\langle \calS \rangle$ of $\calS$ coincide with $\sigma(\ZZ^d)$ and $\pi_i(s)=\sigma_i(\alpha)$ for $s=\sigma(\alpha)\in \calS$.
Since $\pi_i(\langle \calS \rangle)=\sigma_i(\ZZ^d)=\ZZ$ for all $i\in [n]$ and for all $i,j \in [n]$ with $i\neq j$, there exists $s=\sigma(\alpha)\in \sigma(\tau^{\vee}\cap \ZZ^d)$ such that $\pi_j(s)=\sigma_j(\alpha)>0=\sigma_i(\alpha)=\pi_i(s)$, the set $T_i=\{\eb_j+\langle \calS \rangle : i\neq j\in [n]\}$ generates $\ZZ^n/\langle \calS \rangle$ as a semigroup for every $i\in [n]$ (\cite[Theorem 2]{Cho}).
This implies
$$\ZZ_{\ge 0}\beta_1+\cdots+\widehat{\ZZ_{\ge 0}\beta_i}+\cdots+\ZZ_{\ge 0}\beta_n=\ZZ^r$$
for all $i\in [n]$, where $\widehat{\;}$ indicates an element to be omitted. 
Therefore, for any $\nb \in \ZZ^r$, there exists $j\in [n]$ such that $\langle \nb, \beta_j \rangle>0$.
Since $\beta \in \partial\calW'(R)$, there is a facet $F$ of $\calW'(R)$ with $\beta \in F$.
From Lemma~\ref{lem}, $\beta_j$ with $\langle \nb_F, \beta_j \rangle>0$ must appear as a summand in $\beta=\sum_{i\in [n]}a_i\beta_i$, where $\nb_F \in \ZZ^r$ is an outer normal vector of the supporting hyperplane defining $F$. 
Thus, we have $\beta \notin \calW(R)$.

\medskip

(ii) Let $v_1,\ldots,v_s$ be the vertices of $\calC'$. 
Since $v_k\in \calW'(R)$ for each $k\in [s]$, we can write $v_k=\sum_{i=1}^n a_{ki}\beta_i$ for some $a_{ki}\in [0,1]$. 
On the other hand, for any $z\in \calC'$, we can also write $z=\sum_{k=1}^s t_k v_k$ with $t_k\in [0,1]$ and $\sum_{k=1}^st_k=1$. 
Thus, $z=\sum_{k=1}^s t_k(\sum_{i=1}^n a_{ki}\beta_i)=\sum_{i=1}^n(\sum_{k=1}^s t_k a_{ki})\beta_i$. 
Since $\sum_{k=1}^s t_k a_{ki}\in [0,1]$ for all $i\in [n]$, we have $z\in \calW'(R)$, and hence $\calW'(R)=\calC'$.
\end{proof}

\medskip

\subsection{Conic divisorial ideals of Hibi rings}
\label{subsec_conic_Hibi}

In this subsection, we consider the conic divisorial ideals of Hibi rings. 

Let $P$ be a poset such that the Hasse diagram $\calH(\widehat{P})$ of $\widehat{P}$ has $d+1$ vertices and $n$ edges.
For $p \in \widehat{P} \setminus \{\hat{1}\}$, let $U(p)=\{\{p,q\} \in E(\calH(\widehat{P})) : q \text{ covers } p\}$.
Similarly, for $p \in \widehat{P} \setminus \{\hat{0}\}$, let $D(p)=\{\{p,q\} \in E(\calH(\widehat{P})) : q \text{ is covered by } p\}$.

By definition of $\sigma_P(-)$ and (\ref{relation_divisor}), we see that 
the prime divisors $\calD_e$ indexed by the edges $e$ of $\calH(\widehat{P})$ satisfy the relations: 
\begin{align}
\label{relation_div_hibi}
\sum_{e \in U(p)} \calD_e=\sum_{e' \in D(p)} \calD_{e'} \text{ for }p \in \widehat{P} \setminus \{\hat{0}, \hat{1}\}, 
\sum_{e \in U(\hat{0})}\calD_e=0
\; \text{ and }\;  
\sum_{e \in D(\hat{1})}\calD_e=0.
\end{align}
In particular, we can take prime divisors corresponding to edges not contained in a spanning tree as generators of $\Cl(\kk[P])$, 
thus we obtain that $\Cl(\kk[P]) \cong \ZZ^n/\sigma_P(\ZZ^d)\cong  \ZZ^{n-d}$ (\cite[Theorem]{HHN}).
Moreover, any Weil divisor can be described as $\sum_{i=1}^{n-d}a_i\calD_{e_{d+i}}$ and we identify this with $(a_1,\cdots,a_{n-d})\in\ZZ^{n-d}$.

Note that $\calH(\widehat{P})$ can be regarded as a directed graph by orienting the edge $\{p,q\}\in E(\calH(\widehat{P}))$ from $p$ to $q$ if $q$ covers $p$.
In what follows, we fix a spanning tree $T$ of $\calH(\widehat{P})$ and let $e_1,\cdots,e_d$ be its edges. 
Thus, let $e_{d+1},\cdots,e_n$ be the remaining edges of $\calH(\widehat{P})$.
In addition, for $i\in [n-d]$, let $F_i$ be the fundamental cycle of $e_{d+i}$ with respect to $T$, and we assume that $e_{d+i} \in \supp^+(F_i)$. 
For $e\in E(\calH(\widehat{P}))$, let $\beta_e$ be the weight corresponding to the prime divisor $\calD_e$.
Then, $\beta_{e_{d+i}}=\eb_i$ for $i\in [n-d]$ and the other weights $\beta_{e_j}$ for $j\in [d]$ are uniquely determined by the relation (\ref{relation_div_hibi}).

\begin{prop}\label{weight_hibi}
Work with the same notation as above.
Then $\beta_e=\sum_{i\in [n-d]}\vb(F_i)^{(e)}\eb_i$. 
Moreover, for $w=\sum_{e\in E(\calH(\widehat{P}))} a_e\beta_e$ and a cycle $C$ of $\calH(\widehat{P})$, we have
$$\sum_{i\in [n-d]}  \vb(C)^{(e_{d+i})}w^{(i)}=\sum_{e\in E(\calH(\widehat{P}))}a_e \vb(C)^{(e)}.$$
\end{prop}

\begin{proof}
Let $\gamma_e=\sum_{i\in [n-d]}\vb(F_i)^{(e)}\eb_i$ for $e\in E(\calH(\widehat{P}))$ and we show that $\gamma_e=\beta_e$.
We see that $\gamma_{e_{d+j}}=\sum_{i\in [n-d]}\vb(F_i)^{(e_{d+j})}\eb_i=\eb_j=\beta_{e_{d+j}}$ for $j\in [n-d]$.
Thus, it is enough to show that $\gamma_e$'s satisfy the relation (\ref{relation_div_hibi}).
For $p \in \widehat{P} \setminus \{\hat{0}, \hat{1}\}$, let $u(p)=\sum_{e \in U(p)} \gamma_e$ and $d(p)=\sum_{e' \in D(p)} \gamma_{e'}$.
We fix $i\in [n-d]$. If $p\notin V(F_i)$, then $U(p)\cap E(F_i)=\emptyset$ and $D(p)\cap E(F_i)=\emptyset$. Thus, $u(p)^{(i)}=d(p)^{(i)}=0$. 
Suppose that $p \in V(F_i)$. Then, there are exactly two edges $e_1,e_2$ with $e_1,e_2\in E(F_i)$, and only the following two situations may happen:
(a) $e_1,e_2\in U(p)$ or $e_1,e_2\in D(p)$, or (b) $e_1\in U(p)$, $e_2\in D(p)$ or $e_2\in U(p)$, $e_1\in D(p)$. 
In case (a), we can see that $e_1\in \supp^+(F_i), e_2\in \supp^-(F_i)$ or $e_2\in \supp^+(F_i), e_1\in \supp^-(F_i)$, and hence $u(p)^{(i)}=d(p)^{(i)}=0$.
In case (b), we can see that $e_1\in \supp^+(F_i), e_2\in \supp^+(F_i)$ or $e_1\in \supp^-(F_i), e_2\in \supp^-(F_i)$, and hence $u(p)^{(i)}=d(p)^{(i)}=1$ or $-1$.
Therefore, $u(p)=d(p)$ holds. Similarly, we have $u(\hat{0})=d(\hat{1})=0$.
 
Moreover, for $w=\sum_{e\in E(\calH(\widehat{P}))} a_e\beta_e$, we have $w^{(i)}=\sum_{e\in E(\calH(\widehat{P}))}a_e \vb(F_i)^{(e)}$.
Furthermore, we see that $\vb(C)=\sum_{i\in [n-d]}\vb(C)^{(e_{d+i})}\vb(F_i)$ for a cycle $C$ in $\calH(\widehat{P})$ since $\vb(F_1),\ldots,\vb(F_{n-d})$ form bases of the flow space of $\calH(\widehat{P})$.
Therefore, we obtain that 
\begin{align*}
\sum_{i\in [n-d]}  \vb(C)^{(e_{d+i})}w^{(i)} &=\sum_{i\in [n-d]} \vb(C)^{(e_{d+i})}\Big(\sum_{e\in E(\calH(\widehat{P}))}a_e \vb(F_i)^{(e)}\Big) \\
&=\sum_{e\in E(\calH(\widehat{P}))}a_e \Big(\sum_{i\in [n-d]} \vb(C)^{(e_{d+i})}\vb(F_i)^{(e)}\Big)=\sum_{e\in E(\calH(\widehat{P}))}a_e \vb(C)^{(e)}.
\end{align*}
\end{proof}

In what follows, we will give a correspondence of which elements in $\Cl(\kk[P])\cong \ZZ^{n-d}$ describe the conic divisorial ideals.

Let $\calC(P)$ be a convex polytope defined by 
\begin{equation}\label{ccp}
\begin{split}
\calC(P)=\Bigg\{(z_1,\cdots,z_{n-d})& \in \RR^{n-d} : \\
-|\supp^-(C)|+1& \leq \sum_{i\in [n-d]}  \vb(C)^{(e_{d+i})}z_i \leq |\supp^+(C)|-1 \Bigg\},
\end{split}
\end{equation}
where $C$ runs over all circuits in $\calH(\widehat{P})$.
Moreover, let $\calC'(P)$ be a convex polytope defined by 
\begin{equation*}
\begin{split}
\calC'(P)=\Bigg\{(z_1,\cdots,z_{n-d})& \in \RR^{n-d} : \\
-|\supp^-(C')|& \leq \sum_{i\in [n-d]}  \vb(C')^{(e_{d+i})}z_i \leq |\supp^+(C')| \Bigg\},
\end{split}
\end{equation*}
where $C'$ runs over all cycles in $\calH(\widehat{P})$.

\begin{thm}\label{thm:conic_hibi}
Let the notation be the same as above. 
Then, each point in $\calC(P)\cap \ZZ^{n-d}$ one-to-one corresponds to the conic divisorial ideal of $\kk[P]$. 
\end{thm}

\begin{remark}\label{insu}
As mentioned in Section~\ref{sec_main}, this result has already been given in \cite[Theorem 2.4]{HN}.
It was proved that each conic divisorial ideal of $\kk[P]$ corresponds to a point in $\calC(P)\cap \ZZ^{n-d}$ (the set $\calC(P)$ defined in \cite{HN} coincides with that of this paper).
In the fifth step of the proof, it seems to prove the converse, that is, for each point $(m_{d+1},\ldots,m_n)\in \calC(P)\cap\ZZ^{n-d}$, there exist $\bfx \in (-1,0]^d$ and $(\overline{m}_1,\ldots,\overline{m}_n) \in \sigma_P(\ZZ^d)$ such that $m_i-\overline{m}_i=\ulcorner\sigma_{e_i}(\bfx)\urcorner$ for all $i\in [n]$, where $m_1=\cdots =m_d=0$.
But in fact, they have shown that for each point $(m_{d+1},\ldots,m_n)\in \calC(P)\cap\ZZ^{n-d}$ and for each $j\in \{d+1,\ldots,n\}$, there exist $\bfx \in (-1,0]^d$ and $(\overline{m}_1,\ldots,\overline{m}_n) \in \sigma_P(\ZZ^d)$ such that $m_i-\overline{m}_i=\ulcorner\sigma_{e_i}(\bfx)\urcorner$ for $i\in [d]$ and $i=j$, which does not guarantee that $m_i-\overline{m}_i=\ulcorner\sigma_{e_i}(\bfx)\urcorner$ for $i\in \{d+1,\ldots,n\}\setminus \{j\}$.

In this paper, we re-prove this theorem by a different technique.
We derive it by showing $\calW'(\kk[P])=\calC'(P)$ and using Lemma~\ref{conic} (i), rather than showing $\calW(\kk[P])\cap \ZZ^{n-d}=\calC(P)\cap\ZZ^{n-d}$ directly.
\end{remark}

\begin{proof}[Proof of Theorem~\ref{thm:conic_hibi}]
It is enough to show that $\calW(\kk[P])\cap \ZZ^{n-d}=\calC(P)\cap\ZZ^{n-d}$.
We first show that $\calW'(\kk[P])\subset \calC'(P)$. 
Take $w=\sum_{e\in E(\calH(\widehat{P}))}a_e\beta_e \in \calW'(\kk[P])$. 
For each cycle $C$ of $\calH(\widehat{P})$, it follows from Proposition~\ref{weight_hibi} that
\begin{align}\label{supp}
\sum_{i\in [n-d]}  \vb(C)^{(e_{d+i})}w^{(i)}=\sum_{e\in E(\calH(\widehat{P}))}a_e \vb(C)^{(e)}=\sum_{e\in \supp^+(C)}a_e-\sum_{e'\in \supp^-(C)}a_{e'}.
\end{align}
Since $a_e\in [0,1]$, we have  
\begin{align}\label{hyper_ineq}
-|\supp^-(C)| \leq \sum_{i\in [n-d]}  \vb(C)^{(e_{d+i})}w^{(i)} \leq |\supp^+(C)|.
\end{align}
Thus, $w\in \calC'(P)$, and hence $\calW'(\kk[P])\subset \calC'(P)$. 
Moreover, the hyperplanes 
\begin{align}\label{hyper_eq}
\sum_{i\in [n-d]}  \vb(C)^{(e_{d+i})}z_i = |\supp^+(C)| \text{ and } \sum_{i\in [n-d]}  \vb(C)^{(e_{d+i})}z_i = -|\supp^-(C)|
\end{align}
are supporting hyperplanes of $\calC'(P)$ because there exist elements in $\calW'(\kk[P])\subset \calC'(P)$ such that the equality of each side of (\ref{hyper_ineq}) holds respectively.

To show that $\calC'(P)\subset \calW'(\kk[P])$, we prove that any vertex of $\calC'(P)$ is in $\calW'$. 
Since the hyperplanes (\ref{hyper_eq}) support $\calC'(P)$, any vertex $v$ of $\calC'(P)$ can be represented as the intersection of $n-d$ hyperplanes of them.
By reversing the direction of cycles $C_1,\ldots,C_{n-d}$ of $\calH(\widehat{P})$, we may assume that these hyperplanes have the following forms:
\begin{align}\label{equ}
\sum_{i\in [n-d]}  \vb(C_k)^{(e_{d+i})}z_i = |\supp^+(C_k)| \quad \text{ for } k\in [n-d].
\end{align}
From Lemma~\ref{vertex_hibi} below, we have $v\in \calW'$, and hence $\calW'(\kk[P])=\calC'(P)$ by Lemma~\ref{conic} (ii).

Moreover, from Lemma~\ref{conic} (i), we have 
\begin{align*}
\calW(\kk[P])\cap \ZZ^{n-d}=\Bigg\{(z_1,\cdots,z_{n-d})& \in \ZZ^{n-d} : \\
-|\supp^-(C')|+1& \leq \sum_{i\in [n-d]}  \vb(C')^{(e_{d+i})}z_i \leq |\supp^+(C')|-1 \Bigg\},
\end{align*}
where $C'$ runs over all cycles in $\calH(\widehat{P})$.
By the same argument as in the fourth step of the proof of \cite[Theorem 2.4]{HN}, the inequalities arising from a cycle having a chord can be omitted. Therefore, we obtain that $\calW(\kk[P])\cap \ZZ^{n-d}=\calC(P)\cap\ZZ^{n-d}$.
\end{proof}

\begin{lemma}\label{vertex_hibi}
Let $C^+=\bigcup_{k\in [n-d]}\supp^+(C_k)$ and $C^-=\bigcup_{k\in [n-d]}\supp^-(C_k)$. Suppose that the intersection of (\ref{equ}) is a unique point $v=(v_1,\ldots,v_{n-d})\in \RR^{n-d}$. 
Then, $v$ is a vertex of $\calC'(P)$ if and only if $C^+\cap C^-=\emptyset$, in which case $v$ is in $\calW'$.
\end{lemma}

\begin{proof}
It is enough to show that $v\in \calW' \subset \calC'(P)$ (resp. $v\notin \calC'(P)$) if $C^+\cap C^-=\emptyset$ (resp. $C^+\cap C^-\neq \emptyset$). Suppose that $C^+\cap C^-=\emptyset$. Then, $\sum_{e\in C^+}\beta_e \in \calW'$ satisfies (\ref{equ}) for all $k\in [n-d]$. 
In fact, by (\ref{supp}) and $C^+\cap C^-=\emptyset$, we have 
$$\sum_{i\in [n-d]}  \vb(C_k)^{(e_{d+i})}\Big(\sum_{e\in C^+}\beta_e\Big)^{(i)}=\sum_{e\in \supp^+(C_k)}1=|\supp^+(C_k)|.$$
Therefore, $v=\sum_{e\in C^+}\beta_e \in \calW' \subset \calC'(P)$.

If $C^+\cap C^-\neq \emptyset$, there exist $s, t\in [n-d]$ with $\supp^+(C_s)\cap \supp^-(C_t)\neq \emptyset$. We may assume that $s=1$ and $t=2$. We set $\bfu=\vb(C_1)+\vb(C_2)$ and
$$C^*=(\supp^+(C_1)\cap \supp^-(C_2))\cup (\supp^+(C_2)\cap \supp^-(C_1)).$$
Note that $C^*\neq\emptyset$. Since $v$ satisfies (\ref{equ}), we obtain that 

\begin{align}\label{eq1}
\sum_{i\in [n-d]}  \bfu^{(e_{d+i})}v_i&=\sum_{i\in [n-d]}  \Big(\vb(C_1)^{(e_{d+i})}+\vb(C_2)^{(e_{d+i})}\Big)v_i \notag \\
&=|\supp^+(C_1)|+|\supp^+(C_2)| \notag \\
&=|\supp^+(C_1)\setminus C^*|+|\supp^+(C_2)\setminus C^*|+|C^*|.
\end{align}

On the other hand, since $\bfu$ is in the flow space of $\calH(\widehat{P})$, we can write $\bfu=\sum_{i=1}^m \vb(D_i)$, where $D_1,\ldots,D_m$ are cycles of $\calH(\widehat{P})$ with $\supp^+(D_k) \subset (\supp^+(C_1)\cup\supp^+(C_2))\setminus C^*$ and $\supp^-(D_k) \subset (\supp^-(C_1)\cup\supp^-(C_2))\setminus C^*$ for all $k\in [m]$.
This fact follows from a similar argument as in the proof of \cite[Theorem 14.2.2]{GR}. If $v \in \calC'(P)$, then 
\begin{align*}
\sum_{i\in [n-d]}  \vb(D_k)^{(e_{d+i})}v_i \le |\supp^+(D_k)| \quad \text{ for all } k\in [m]. 
\end{align*}
Thus, we have 
\begin{align*}
\sum_{i\in [n-d]}  \bfu^{(e_{d+i})}v_i &=\sum_{k\in [m]}\Big(\sum_{i\in [n-d]}  \vb(D_k)^{(e_{d+i})}v_i\Big) \\
&\le \sum_{k\in [m]}|\supp^+(D_k)|=|\supp^+(C_1)\setminus C^*|+|\supp^+(C_2)\setminus C^*|,
\end{align*}
a contradiction to (\ref{eq1}). Hence, we obtain that $v\notin \calC'(P)$.
\end{proof}

\medskip

\subsection{Conic divisorial ideals of stable set rings}\label{sec_conic_stab}
In this subsection, we consider conic divisorial ideals of stable set rings of perfect graphs.

In what follows, let $G$ be a perfect graph on the vertex set $V(G)=[d]$ which has maximal cliques $Q_0,Q_1,\ldots,Q_n$.
By definition of $\sigma_G(-)$ and (\ref{relation_divisor}), we see that 
the prime divisor $\calD_i$ for $i\in \{0,1,\dots,n+d\}$ satisfies the relations: 
\begin{align}
\label{relation_div_stab}
\sum_{j=0}^n \calD_j=0 \quad
\; \text{ and }\;  \quad
\calD_{n+k}=\sum_{j=0}^n \chi_j(k)\calD_j \; \text{ for } k\in [d],
\end{align}
where
\begin{align*}
\chi_j(k) =
\begin{cases}
1 \; &\text{ if $k\in Q_j$}, \\
0 \; &\text{ if $k\notin Q_j$}.
\end{cases}
\end{align*}
In particular, we can see that prime divisors $\calD_1,\ldots,\calD_n$ generate $\Cl(\kk[\Stab_G])$, 
thus we have that $\Cl(\kk[\Stab_G]) \cong \ZZ^{n+d+1}/\sigma_G(\ZZ^{d+1})\cong  \ZZ^{n}$ (see \cite{HM3}).
Furthermore, let $\beta_i$ be the weight corresponding to the prime divisor $\calD_i$.
Then, we can determine the weights $\beta_i$ ($i\in\{0,1,\ldots,n+d\}$) by the relation (\ref{relation_div_stab}):
\begin{align}\label{weight_stab}
\beta_i =
\begin{cases}
\eb_i \; &\text{ if $i\in \{0,1,\ldots,n\}$}, \\
\sum_{j=0}^n \chi_j(k)\eb_j \; &\text{ if $i\in \{n+1,\ldots,n+d\}$}, \\
\end{cases}
\end{align}
where we let $\eb_0=-\eb_1-\cdots -\eb_n$. For $v\in V(G)$ and a finite multiset $L\subset \{0,1,\ldots,n\}$, let $m_L(v)= |\{l\in L : v\in Q_l\}|$.
Moreover, for finite multisets $I,J\subset \{0,1,\ldots,n\}$, we set
\begin{align}\label{XIJ}
X_{IJ}^+= \{v\in V(G) : m_{IJ}(v) >0\} \text{ and } X_{IJ}^-= \{v\in V(G) : m_{IJ}(v) <0\},
\end{align}
where $m_{IJ}(v)= m_I(v)-m_J(v)$.

\medskip

Let $\calC(G)$ and $\calC'(G)$ be two convex polytopes defined by
\begin{equation}\label{conic_stab}
\begin{split}
\calC(G)=\Bigg\{(&z_1,\cdots,z_{n}) \in \RR^n : \\ 
-&|J|+\sum_{v\in X_{IJ}^-}m_{IJ}(v)+1 \leq \sum_{i\in I}z_i -\sum_{j\in J}z_j \leq |I|+\sum_{v\in X_{IJ}^+}m_{IJ}(v)-1 \\
 & \text{ for finite multisets } I, J\subset \{0,1,\ldots,n\} \text{ with } |I|=|J| \text{ and } I\cap J=\emptyset \Bigg\}\text{ and }
\end{split}
\end{equation}
 
\begin{equation}\label{conic_stab'}
\begin{split}
\calC'(G)=\Bigg\{(z_1&,\cdots,z_{n}) \in \RR^n : \\
-&|J|+\sum_{v\in X_{IJ}^-}m_{IJ}(v) \leq \sum_{i\in I}z_i -\sum_{j\in J}z_j \leq |I|+\sum_{v\in X_{IJ}^+}m_{IJ}(v) \\
& \text{ for finite multisets } I, J\subset \{0,1,\ldots,n\} \text{ with } |I|=|J| \text{ and } I\cap J=\emptyset \Bigg\},
\end{split}
\end{equation}
where we let $z_0=0$.

\begin{remark}\label{rem_sta}
Note that an infinite number of inequalities appear in (\ref{conic_stab}) and (\ref{conic_stab'}).
But in fact, only finitely many inequalities are needed since it follows from Theorem~\ref{thm:conic_stab} below that $\calC'(G)$ coincides with $\calW'(\kk[\Stab_G])$. 
Therefore, $\calC(G)$ and $\calC'(G)$ are polytopes.
On the other hand, by using Lemma~\ref{lem}, we can determine the facet defining inequalities of $\calC'(G)$.
For example, since $\eb_0,\eb_1,\ldots,\eb_n$ must appear as a weight of $\kk[\Stab_G]$, for each $i,j\in \{0,1,\ldots,n\}$ with $i\neq j$, the inequality 
$$-1-|X_{\{i\}\{j\}}^-| \leq z_i -z_j \leq 1+|X_{\{i\}\{j\}}^+|$$
defines a facet of $\calC'(G)$.
\end{remark}

\begin{ex}
Let $\Gamma$ be the graph on the vertex set $\{1,\ldots,7\}$ with the edge set
$$E(\Gamma)=\{12,13,23,24,25,34,36,45,46,56,57,67\}.$$
See Figure~\ref{Gamma}. 

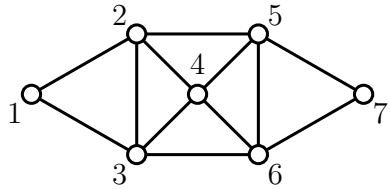
\begin{figure}[h]
\centering
{\scalebox{0.8}{
\begin{tikzpicture}[line width=0.05cm]

\coordinate (3) at (0,0); \coordinate (1) at (-1.73,1); \coordinate (2) at (0,2); 
\coordinate (5) at (2,2); \coordinate (4) at (1,1); \coordinate (6) at (2,0); 
\coordinate (7) at (3.73,1); 

\draw (1)--(2); \draw (1)--(3); \draw (2)--(3); \draw (2)--(4);
\draw (2)--(5); \draw (3)--(4); \draw (3)--(6); \draw (4)--(5);
\draw (4)--(6); \draw (5)--(6); \draw (5)--(7); \draw (6)--(7);


\draw [line width=0.05cm, fill=white] (1) circle [radius=0.15] node[below left] {\Large 1}; 
\draw [line width=0.05cm, fill=white] (2) circle [radius=0.15] node[above left] {\Large 2};
\draw [line width=0.05cm, fill=white] (3) circle [radius=0.15] node[below left] {\Large 3};
\draw [line width=0.05cm, fill=white] (4) circle [radius=0.15] node[] at (1,1.5) {\Large 4};
\draw [line width=0.05cm, fill=white] (5) circle [radius=0.15] node[above right] {\Large 5};
\draw [line width=0.05cm, fill=white] (6) circle [radius=0.15] node[below right] {\Large 6};
\draw [line width=0.05cm, fill=white] (7) circle [radius=0.15] node[below right] {\Large 7};


\end{tikzpicture}
}}
\caption{The graph $\Gamma$}
\label{Gamma}

\end{figure}

\noindent Then, $\Gamma$ is a perfect graph and has $6$ maximal cliques:
$$Q_0=\{1,2,3\}, Q_1=\{2,3,4\},Q_2=\{2,4,5\},Q_3=\{3,4,6\},Q_4=\{4,5,6\} \text{ and } Q_5=\{5,6,7\}.$$ 
Let $I=\{1,1,5\}$ and $J=\{0,2,3\}$. Then, we have 
$$m_{IJ}(1)=-1, \; m_{IJ}(2)=m_{IJ}(3)=m_{IJ}(4)=m_{IJ}(5)=m_{IJ}(6)=0 \; \text{ and } \; m_{IJ}(7)=1.$$
Thus, we obtain that $X_{IJ}^+=\{7\}$ and $X_{IJ}^-=\{1\}$. Therefore, we get the inequality
$$-4\le 2z_1+z_5 -z_2-z_3 \le 4.$$
Indeed, this inequality is a facet defining inequality of $\calC'(\Gamma)$.
\end{ex}

\begin{thm}\label{thm:conic_stab}
Let the notation be the same as above. 
Then, each point in $\calC(G)\cap \ZZ^n$ one-to-one corresponds to the conic divisorial ideal of $\kk[\Stab_G]$. 
\end{thm}

\begin{proof}
We prove that $\calW(\kk[\Stab_G])\cap\ZZ^n=\calC(G)\cap\ZZ^n$ by the same discussion as in the case of Hibi rings.
We take $w=\sum_{i=0}^{n+d}a_i\beta_i \in \calW'(\kk[\Stab_G])$. It follows from (\ref{weight_stab}) that 
\begin{align*}
w^{(i)}=a_i+\sum_{v\in Q_i}a_{v+n} -\sum_{u\in Q_0}a_{u+n} -a_0.
\end{align*}
Therefore, for finite multisets $I,J\subset \{0,1,\ldots,n\}$ with $|I|=|J|$ and $I\cap J=\emptyset$,
\begin{align}\label{supp_stab}
\sum_{i\in I}w^{(i)}-\sum_{j\in J}w^{(j)}&=\sum_{i\in I}\Big(a_i+\sum_{v\in Q_i}a_{v+n} -\sum_{u\in Q_0}a_{u+n} -a_0\Big) \notag \\
&\hspace{4.0cm}-\sum_{j\in J}\Big(a_j+\sum_{v\in Q_j}a_{v+n} -\sum_{u\in Q_0}a_{u+n} -a_0\Big) \notag \\
&=\sum_{i\in I}a_i + \sum_{v=1}^{d}m_I(v)a_{v+n} - \sum_{u=1}^{d}m_J(u)a_{u+n} -\sum_{j\in J}a_j \notag \\
&=\sum_{i\in I}a_i -\sum_{j\in J}a_j +\sum_{v=1}^{d}m_{IJ}(v)a_{v+n}.
\end{align}
Since $a_i\in [0,1]$, we have 
\begin{align*}
-|J|+\sum_{v\in X_{IJ}^-}m_{IJ}(v) \leq \sum_{i\in I}w^{(i)} -\sum_{j\in J}w^{(j)} \leq |I|+\sum_{v\in X_{IJ}^+}m_{IJ}(v).
\end{align*}
Thus, $w\in \calC'(G)$, and hence $\calW'(\kk[\Stab_G])\subset \calC'(G)$. 
Furthermore, the hyperplanes 
\begin{align}\label{hyper_eq_stab}
\sum_{i\in I}z_i -\sum_{j\in J}z_j = |I|+\sum_{v\in X_{IJ}^+}m_{IJ}(v) \text{ and } \sum_{i\in I}z_i -\sum_{j\in J}z_j =-|J|+\sum_{v\in X_{IJ}^-}m_{IJ}(v).
\end{align}
are supporting hyperplanes of $\calC'(G)$.

Next, we prove that any vertex of $\calC'(G)$ is in $\calW'(\kk[\Stab_G])$. 
We consider $n$ supporting hyperplanes (\ref{hyper_eq_stab}) whose intersection is a unique point $u$.
By alternating $I_k$ and $J_k$, we may assume that these hyperplanes have the following forms:
\begin{align}\label{equ_stab}
\sum_{i\in I_k}w^{(i)} -\sum_{j\in J_k}w^{(j)} = |I_k|+\sum_{v\in X_{I_kJ_k}^+}m_{I_kJ_k}(v) \quad \text{ for } k\in [n].
\end{align}
From the following lemma, we have $u\in \calW'(\kk[\Stab_G])$, and hence $\calW(\kk[\Stab_G])\cap\ZZ^n=\calC(G)\cap\ZZ^n$.
\end{proof}

\begin{lemma}\label{vertex_stab}
Let $X^+=\bigcup_{k\in [n]}X_{I_kJ_k}^+$ and $X^-=\bigcup_{k\in [n]}X_{I_kJ_k}^-$. Suppose that the intersection of (\ref{equ_stab}) is a unique point $u=(u_1,\ldots,u_n)\in \RR^{n}$. 
Then, $u$ is a vertex of $\calC'(G)$ if and only if $X^+\cap X^-=\emptyset$, in which case $u$ is in $\calW'(\kk[\Stab_G])$.
\end{lemma}

\begin{proof}
It is enough to show that $u\in \calW' \subset \calC'(G)$ (resp. $u\notin \calC'(G)$) if $X^+\cap X^-=\emptyset$ (resp. $X^+\cap X^-\neq \emptyset$). Suppose that $X^+\cap X^-=\emptyset$. 
From (\ref{supp_stab}), we can see that $\sum_{l=0}^{n+d} \alpha_l\beta_l\in \calW'$ satisfies (\ref{equ_stab}) for all $k\in [n]$, where
\begin{align*}
\alpha_l=\begin{cases}
1 \; \text{ if }l\in \bigcup_{k=1}^n I_k \text{ or } l-n\in X^+, \\
0 \; \text{ otherwise. } \end{cases}
\end{align*} 
Therefore, we have $u=\sum_{l=0}^{n+d} \alpha_l\beta_l \in \calW' \subset \calC'(P)$.

\medskip

If $X^+\cap X^-\neq \emptyset$, there exist $s, t\in [n]$ with $X_{I_sJ_s}^+\cap X_{I_tJ_t}^-\neq \emptyset$. We may assume that $s=1$ and $t=2$. We set $I'=(I_1\cup I_2)\setminus (J_1\cup J_2)$, $J'=(J_1\cup J_2)\setminus (I_1\cup I_2)$,
$$X_1^*=X_{I_1J_1}^+\cap X_{I_2J_2}^-\text{ and }X_2^*= X_{I_2J_2}^+\cap X_{I_1J_1}^-.$$
Here, $I'$ and $J'$ are regarded as multisets.
Note that $X_1^*\neq\emptyset$, $X_{I'J'}^+\cup X_1^*\cup X_2^*=X_1^+\cup X_2^+$ and $m_{I'J'}(v)=m_{I_1J_1}(v)+m_{I_2J_2}(v)$ for all $v\in V(G)$.
Since $u$ satisfies (\ref{equ}), we obtain that
\begin{align}\label{eq1_stab}
\sum_{i\in I'}u_i-\sum_{j\in J'}u_j &= \Big(\sum_{i\in I_1}u_i-\sum_{j\in J_1}u_j \Big)+\Big(\sum_{i'\in I_2}u_{i'}-\sum_{j\in J_2}u_{j'} \Big)\notag \\
&=|I_1|+\sum_{v\in X_{I_1J_1}^+}m_{I_1J_1}(v)+|I_2|+\sum_{v'\in X_{I_2J_2}^+}m_{I_2J_2}(v')\notag \\
&=|I_1|+|I_2|+\sum_{v\in X_{I'J'}^+\setminus (X_1^*\cup X_2^*)}m_{I'J'}(v)+\sum_{v\in X_1^*}m_{I_1J_1}(v)+\sum_{v\in X_2^*}m_{I_2J_2}(v).
\end{align}

On the other hand, since $|I'|=|J'|$ and $I'\cap J'=\emptyset$, we have
\begin{align*}
\sum_{i\in I'}u_i-\sum_{j\in J'}u_j \le |I'|+\sum_{v\in X_{I'J'}^+}m_{I'J'}(v) 
\end{align*}
if $u\in \calC'(G)$, a contradiction to (\ref{eq1_stab}) because $|I'|\le |I_1|+|I_2|$ and $m_{I'J'}(v)<m_{I_1J_1}(v)$ (resp. $m_{I'J'}(v)<m_{I_2J_2}(v)$) for $v\in X_1^*$ (resp. $v\in X_2^*$).  
Thus, we obtain that $v\notin \calC'(G)$.

\end{proof}

We finally state the description of conic divisorial ideals of the stable set ring arising from the comparability graph of a poset $P$.
In this case, we expect to be able to describe them in terms of $P$ as follows.

Let $P$ be a poset and let $\calQ(\widehat{P})$ denote the set of maximal chains of $\widehat{P}$.
Moreover, let $C=(p_1,\ldots,p_s)$ be a cycle of $\calH(\widehat{P})$ and we may assume that
$$p_{m_1}=p_1\prec p_2 \prec \cdots \prec p_{M_1} \succ \cdots \succ p_{m_2} \prec \quad \cdots \quad \prec p_{M_k} \succ \cdots \succ p_s \succ p_{m_{k+1}}=p_1.$$
We set ${\bf U}_C=\calQ(\widehat{P}_{\succeq p_{M_1}})\times \cdots \times \calQ(\widehat{P}_{\succeq p_{M_k}})$ and ${\bf D}_C=\calQ(\widehat{P}_{\preceq p_{m_1}})\times \cdots \times \calQ(\widehat{P}_{\preceq p_{m_k}})$, where $\widehat{P}_{\succeq p}=\{q\in \widehat{P} : q\succeq p\}$ for $p\in \widehat{P}$ (we define $\widehat{P}_{\preceq p}$ analogously).
Furthermore, for $U=(U_1,\ldots,U_k)\in {\bf U}_C$ and $D=(D_1,\ldots,D_k)\in {\bf D}_C$, the sets
$$Q_i^{\uparrow}=D_i\cup \{p_{m_i},p_{m_i+1}\ldots,p_{M_i}\}\cup U_i \quad \text{ and } \quad Q_i^{\downarrow}=U_i\cup \{p_{M_i},p_{M_i+1}\ldots,p_{m_{i+1}}\}\cup D_{i+1}$$
are maximal chains of $\widehat{P}$ for each $i\in [k]$, where $U_{k+1}=U_1$. Fix $Q_0\in \calQ(\widehat{P})$ and we define 
\begin{multline*}
\fkC_P=\Bigg\{z \in \RR^{\calQ(\widehat{P})\setminus \{Q_0\}} : \\ 
-k-\sum_{l=1}^k (m_{l+1}-M_l-1)+1 \leq \sum_{i\in 1}^k z^{(Q_i^{\uparrow})} -\sum_{j\in 1}^k z^{(Q_j^{\downarrow})} \leq k+\sum_{l=1}^k (M_l-m_l-1) -1\\
 \text{ for $U\in {\bf U}_C$ and $D\in {\bf D}_C$} \Bigg\},
\end{multline*}
where $C$ runs over all circuits in $\calH(\widehat{P})$ and we let $m_{k+1}=s+1$ and $z^{(Q_0)}=0$.

\medskip

We call a poset $\calX$ \textit{general X-shape} if $\calX$ is the ordinal sum of some chains and some disjoint unions of two chains.
The Hasse diagram of a general X-shape poset looks like the one shown in Figure~\ref{poset}. 
Suppose that $\widehat{P}$ contains a general X-shape subposet $\calX$ with $\calQ(\calX)\subset \calQ(\widehat{P})$ and $|\calQ(\calX)|\ge 4$.
For $Q\in \calQ(\calX)$, there exists $\overline{Q}\in \calQ(\calX)$ with $Q\cup \overline{Q}=\calX$, which is uniquely determined.
We set 
\begin{align*}
\fkX_P=\Bigg\{z \in \RR^{\calQ(\widehat{P})\setminus \{Q_0\}} :  
-1 \leq z^{(Q)}+z^{(\overline{Q})} -z^{(Q')}-z^{(\overline{Q'})} \leq 1 
 \text{ for $Q, Q' \in \calQ(\calX)$} \Bigg\},
\end{align*}
where $\calX$ runs over all general X-shape subposet of $\widehat{P}$ with $\calQ(\calX)\subset \calQ(\widehat{P})$ and $|\calQ(\calX)|\ge 4$.

\begin{ex}
Let $\Pi=\{p_1,\ldots,p_6\}$ be the poset which is the ordinal sum of the disjoint union of two elements and the disjoint union of three elements.
The Hasse diagram of $\widehat{\Pi}$ is shown in Figure~\ref{Pi}.

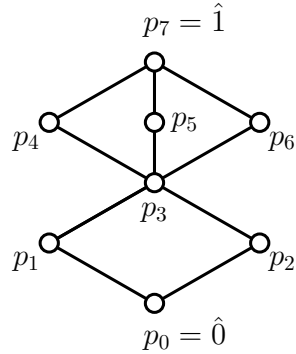
\begin{figure}[h]
\centering
{\scalebox{0.8}{
\begin{tikzpicture}[line width=0.05cm]

\coordinate (1) at (0,0); \coordinate (2) at (-1.73,1); \coordinate (3) at (1.73,1); 
\coordinate (4) at (0,2); \coordinate (5) at (-1.73,3); \coordinate (6) at (0,3);
\coordinate (7) at (1.73,3); \coordinate (8) at (0,4); 


\draw (1)--(2); \draw (1)--(3); \draw (2)--(4); \draw (3)--(4);
\draw (4)--(5); \draw (4)--(6); \draw (4)--(7); \draw (2)--(4);
\draw (5)--(8); \draw (6)--(8); \draw (7)--(8);


\draw [line width=0.05cm, fill=white] (1) circle [radius=0.15] node[] at (0.5,-0.5) {\Large $p_0=\hat{0}$}; 
\draw [line width=0.05cm, fill=white] (2) circle [radius=0.15] node[below left] {\Large $p_1$};
\draw [line width=0.05cm, fill=white] (3) circle [radius=0.15] node[below right] {\Large $p_2$};
\draw [line width=0.05cm, fill=white] (4) circle [radius=0.15] node[] at (0,1.5){\Large $p_3$};
\draw [line width=0.05cm, fill=white] (5) circle [radius=0.15] node[below left] {\Large $p_4$};
\draw [line width=0.05cm, fill=white] (6) circle [radius=0.15] node[] at (0.5,3) {\Large $p_5$};
\draw [line width=0.05cm, fill=white] (7) circle [radius=0.15] node[below right] {\Large $p_6$};
\draw [line width=0.05cm, fill=white] (8) circle [radius=0.15] node[] at (0.5,4.7) {\Large $p_7=\hat{1}$};


\end{tikzpicture}
}}
\caption{The Hasse diagram of $\widehat{\Pi}$}
\label{Pi}
\end{figure}

\noindent It has $6$ maximal chains, $4$ circuits and $3$ general X-shape subposets satisfying the appropriate conditions.
Let $C=(p_3,p_4,p_7,p_5)$ be a circuit of $\calH(\widehat{\Pi})$.
In this case, we can see that 
$${\bf U}_C=\calQ(\widehat{\Pi}_{\succeq p_7})=\{\{p_7\}\} \quad \text{ and } \quad {\bf D}_C=\calQ(\widehat{\Pi}_{\preceq p_3})=\{\{p_0,p_1,p_3\},\{p_0,p_2,p_3\}\}.$$
Let $U=U_1=\{p_7\}\in {\bf U}_C$ and $D=D_1=\{p_0,p_1,p_3\}\in {\bf D}_C$. Then, we have 
$$Q_1^{\uparrow}=\{p_0,p_1,p_3,p_4,p_7\} \quad \text{ and } \quad Q_1^{\downarrow}=\{p_7,p_5,p_3,p_1,p_0\},$$
and hence we get the inequality
$$-1\le z^{(Q_1^{\uparrow})}-z^{(Q_1^{\downarrow})}\le 1.$$

Next, we consider the general X-shape subposet $\calX=\widehat{\Pi}\setminus \{p_5\}$. 
Let $Q=\{p_0,p_1,p_3,p_4,p_7\}$ and $Q'=\{p_0,p_2,p_3,p_4,p_7\}$. Then, we have
$$\overline{Q}=\{p_0,p_2,p_3,p_6,p_7\} \quad \text{ and } \quad \overline{Q'}=\{p_0,p_1,p_3,p_6,p_7\},$$
and hence we get the inequality
$$-1\le z^{(Q)}+z^{(\overline{Q})}-z^{(Q')}-z^{(\overline{Q'})}\le 1.$$
\end{ex}

\begin{conj}
Let $P$ be a poset. Then, the conic divisorial ideals of $\kk[\calC_P]$ one-to-one correspond to the points in $\fkC_P\cap \fkX_P\cap \ZZ^{\calQ(\widehat{P})\setminus \{Q_0\}}$.
\end{conj}

We can see that the inequalities appearing in $\fkC_P$ and $\fkX_P$ are special forms of those appearing in $\calC(G(P))$.
Therefore, we have $\calC(G(P)) \subset \fkC_P\cap \fkX_P$, that is, each conic divisorial ideal of $\kk[\calC_P]$ corresponds to a point in $\fkC_P\cap \fkX_P\cap \ZZ^{\calQ(\widehat{P})\setminus \{Q_0\}}$.
We expect that the converse is true.

\bigskip

\section{The characterization of quasi-symmetric or weakly-symmetric toric rings}\label{sec_qw}

In this section, we characterize when our toric rings are quasi-symmetric or weakly-symmetric, that is, we prove Theorems~\ref{main2_hibi} and \ref{main2_stab}.

\medskip

\subsection{Proof of Theorem~\ref{main2_hibi}}

\begin{proof}
Let $P$ be a poset such that $\calH(\widehat{P})$ has $d+1$ vertices and $n$ edges.

(i) $\Rightarrow$ (ii): This follows immediately from Proposition~\ref{p_sum} (i) and (ii).

(ii) $\Rightarrow$ (iii): If $\kk[P]$ is isomorphic to the tensor product of a polynomial ring and some Segre products of two polynomial rings, then it is also isomorphic to the Hibi ring of a general X-shape poset $\calX$.
We can easily compute the weights of $\kk[P] \cong \kk[\calX]$ by using Proposition~\ref{weight_hibi}.
In fact, since the circuits of $\calH(\widehat{\calX})$ are precisely the fundamental cycles of $\calH(\widehat{\calX})$ and their edge sets are disjoint, the weights of $\kk[P]$ is $\pm\eb_1,\ldots,\pm\eb_{n-d}$ or $0$, and $\pm\eb_1,\ldots,\pm\eb_{n-d}$ must appear.
Therefore, $\kk[P]$ is weakly-symmetric.

Before proving (iii) $\Rightarrow$ (i), we give an easy observation: $P$ has an element which is comparable with any other element of $P$ if and only if $\calH(\widehat{P})$ is not $2$-connected, i.e., there exists an element $p$ in $P$ such that $\calH(\widehat{P})\setminus p$ is not connected.
In this case, we can see that $P=P_{\prec p}\oplus P_{\succ p}$.
Therefore, we may assume that $\calH(\widehat{P})$ is 2-connected by Proposition~\ref{p_sum} (ii).

(iii) $\Rightarrow$ (i): Since $\calH(\widehat{P})$ is 2-connected, it can be constructed from the Hasse diagram $H$ (see Figure~\ref{poset1}) by successively adding paths to graphs already constructed (see, e.g., \cite[Proposition 3.1.1]{Die}). 
In this case, we can replace ``paths'' with ``chains''.
Moreover, by removing an edge from $H$ and each added chain, we get a spanning tree $T$ of $\calH(\widehat{P})$. We denote those edges by $e_1,\ldots,e_{n-d}$ and assume that $e_1\in E(H)$ and $e_2$ is in the added chain to the first. 
For $i\in [n-d]$, let $F_i$ be the fundamental cycle of $e_i$ with respect to $T$. Note that $F_1=H$.

Since $\kk[P]$ is weakly-symmetric and any weight of $\kk[P]$ is in $\{0,1,-1\}^{n-d}$, there is the weight $-\eb_2$. 
However, it is impossible because any edge of $F_2$ is contained in $E(F_1)$ or the edge set of the added chain to the first, in particular, $\supp^-(F_2) \subset E(F_1)$. 
Thus, $\calH(\widehat{P})=H$.

Suppose that $\kk[P]$ is Gorenstein.
If the condition (i) is satisfied, then we can compute the weights of $\kk[P]$ as in (ii) $\Rightarrow$ (iii) and check that $\kk[P]$ is quasi-symmetric since $P$ is pure. 
Clearly, (iv) $\Rightarrow$ (iii) holds, and hence those four conditions are equivalent. 
\end{proof}

\begin{figure}[h]
\centering
{\scalebox{0.8}{
\begin{tikzpicture}[line width=0.05cm]

\coordinate (N11) at (0,0); \coordinate (N12) at (0,1); \coordinate (N13) at (0,2); 
\coordinate (N14) at (0,3); \coordinate (N15) at (0,4); \coordinate (N16) at (0,5);
\coordinate (N21) at (3,0); \coordinate (N22) at (3,1); \coordinate (N23) at (3,2); 
\coordinate (N24) at (3,3); \coordinate (N25) at (3,4); \coordinate (N26) at (3,5);
\coordinate (M1) at (1.5,-1.0); \coordinate (M8) at (1.5,5.0);

\draw  (N11)--(N12); \draw  (N12)--(0,1.5); \draw  (N14)--(0,2.5);
\draw  (N14)--(N15); \draw[dotted]  (0,1.8)--(0,2.2);
\draw  (N21)--(N22);\draw  (N22)--(3,1.5); \draw  (N24)--(3,2.5);
\draw  (N24)--(N25); \draw[dotted]  (3,1.8)--(3,2.2);
\draw  (N11)--(M1); \draw  (N21)--(M1); \draw  (M8)--(N15); \draw  (M8)--(N25);


\draw [line width=0.05cm, fill=white] (N11) circle [radius=0.15]; 
\draw [line width=0.05cm, fill=white] (N12) circle [radius=0.15]; 
\draw [line width=0.05cm, fill=white] (N14) circle [radius=0.15]; 
\draw [line width=0.05cm, fill=white] (N15) circle [radius=0.15]; 
\draw [line width=0.05cm, fill=white] (N21) circle [radius=0.15]; 
\draw [line width=0.05cm, fill=white] (N22) circle [radius=0.15]; 
\draw [line width=0.05cm, fill=white] (N24) circle [radius=0.15]; 
\draw [line width=0.05cm, fill=white] (N25) circle [radius=0.15]; 
\draw [line width=0.05cm, fill=white] (M1) circle [radius=0.15]; 
\draw [line width=0.05cm, fill=white] (M8) circle [radius=0.15];

\end{tikzpicture}
} }
\caption{The Hasse diagram $H$}
\label{poset1}
\end{figure}
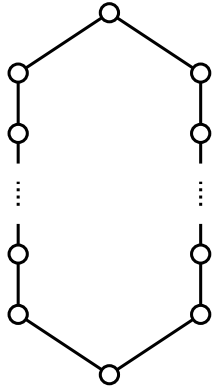


\medskip

\subsection{Proof of Theorem~\ref{main2_stab}}

\begin{proof}
Let $G$ be a perfect graph with maximal cliques $Q_0,Q_1,\ldots,Q_n$. 

(i) $\Rightarrow$ (ii): In the case $n=0$, $\kk[\Stab_G]$ is the polynomial ring with $|Q_0|+1$ variables over $\kk$. 
Suppose that $n=1$ and let $Q=Q_0\cap Q_1$.
Note that for each $v\in Q$, we have $\{v,w\}\in E(G)$ for all $w\in V(G)$.
By observing stable sets of $G$, we can see that $\kk[\Stab_G]\cong \kk[\Stab_{G\setminus Q}]\otimes_{\kk}\kk[\Stab_{G_Q}]$ and $\kk[\Stab_{G\setminus Q}]$ is isomorphic to the Segre product of $\kk[\Stab_{G_{Q_0\setminus Q}}]$ and $\kk[\Stab_{G_{Q_1\setminus Q}}]$.
Furthermore, $\kk[\Stab_{G_Q}]$, $\kk[\Stab_{G_{Q_0\setminus Q}}]$ and $\kk[\Stab_{G_{Q_1\setminus Q}}]$ are polynomial rings.
Thus, $\kk[\Stab_G]$ is isomorphic to the tensor product of a polynomial ring and the Segre products of two polynomial rings

(ii) $\Rightarrow$ (iii): This is the same as in the case of Hibi rings.

(iii) $\Rightarrow$ (i): Since $\beta_i \in \{0,1,-1\}^n$ for any $i\in \{0,1,\ldots,n+d\}$ and $\kk[\Stab_G]$ is weakly-symmetric, the weights $-\eb_k$ ($k\in \{0,1,\ldots,n\}$) must appear. 
Equivalently, for each $j \in \{0,1,\ldots,n\}$, there exists $v_j \in V(G)$ such that $v_j\notin Q_j$ and $v_j\in Q_l$ for any $l\in \{0,1,\ldots,n\}\setminus \{j\}$.

If $n\ge 2$, then $\{v_s,v_t\}\in E(G)$ for any $s,t\in \{0,1,\ldots,n\}$ because there exists $u\in \{0,1,\ldots,n\}$ with $v_s,v_t\in Q_u$. Therefore, $\{v_0,v_1,\ldots,v_n\}$ is a clique of $G$, and a maximal clique containing it is different from $Q_0,Q_1,\ldots,Q_n$. 
Hence, we have $n\le 1$.

If $\kk[\Stab_G]$ is Gorenstein, then the conditions (i), (ii), (iii), and (iv) are equivalent by the same argument as in the case of Hibi rings.
\end{proof}

\bigskip

\section{Construction of NCCRs for a special family of stable set rings}\label{sec_nccr}

Finally, in this section, we introduce a perfect graph $G_{r_1,\ldots,r_n}$ and give an NCCR for its stable set ring, i.e., 
we prove Theorem~\ref{main3}.

\medskip

\subsection{Perfect graphs $G_{r_1,\ldots,r_n}$}\label{family}

For an integer $n\ge 3$ and positive integers $r_1,\ldots,r_n$,
let $G_{r_1,\ldots,r_n}$ be the graph on the vertex set $V(G_{r_1,\ldots,r_n})=[2d]$ with the edge set $E(G_{r_1,\ldots,r_n})=\bigcup_{i=0}^n \big\{\{v,u\} : v,u\in Q_i\big\}$, where $d= \sum_{k=1}^n r_k$, $Q_0=\{d+1,\ldots,2d\}$ and for $i\in [n]$, we let
\begin{align*}
&Q_i^+=\Big\{r_1+\cdots+r_{i-1}+1,\ldots,r_1+\cdots+r_{i-1}+r_i\Big\},\\ &Q_i^-=\Big\{d+r_1+\cdots+r_{i-1}+1,\ldots,d+r_1+\cdots+r_{i-1}+r_i\Big\}\text{ and }\\
&Q_i=Q_i^+\cup (Q_0\setminus Q_i^-).
\end{align*}
Note that $Q_i^+=Q_i\setminus Q_0$ and  $Q_i^-=Q_0\setminus Q_i$.

\begin{ex}
We give drawings of $G_{1,1,1}$ and $G_{1,1,1,1}$ in Figures~\ref{G111} and \ref{G1111}, respectively.

\begin{figure}[h]

\begin{minipage}{0.485\columnwidth}
\centering
{\scalebox{0.75}{
\begin{tikzpicture}[line width=0.05cm]

\coordinate (1) at (0,0); \coordinate (6) at (-1.73,1); \coordinate (2) at (-1.73,3); 
\coordinate (4) at (0,4); \coordinate (3) at (1.73,3); \coordinate (5) at (1.73,1); 


\draw (1)--(5); \draw (1)--(6); \draw (2)--(4); \draw (2)--(6);
\draw (3)--(4); \draw (3)--(5); \draw (4)--(5); \draw (4)--(6);
\draw (5)--(6);


\draw [line width=0.05cm, fill=white] (1) circle [radius=0.15] node[] at (0,-0.5) {\Large 1}; 
\draw [line width=0.05cm, fill=white] (2) circle [radius=0.15] node[above left] {\Large 2};
\draw [line width=0.05cm, fill=white] (3) circle [radius=0.15] node[above right] {\Large 3};
\draw [line width=0.05cm, fill=white] (4) circle [radius=0.15] node[] at (0,4.5) {\Large 4};
\draw [line width=0.05cm, fill=white] (5) circle [radius=0.15] node[below right] {\Large 5};
\draw [line width=0.05cm, fill=white] (6) circle [radius=0.15] node[below left] {\Large 6};

\node at (0,5) {$\;$};

\end{tikzpicture}
}}
\caption{The graph $G_{1,1,1}$}
\label{G111}
\end{minipage}
\begin{minipage}{0.5\columnwidth}
\centering
{\scalebox{0.6}{
\begin{tikzpicture}[line width=0.05cm]

\coordinate (7) at (0,0); \coordinate (1) at (0,2); \coordinate (3) at (0,4); 
\coordinate (5) at (0,6); \coordinate (6) at (-3,3); \coordinate (4) at (-1,3);
\coordinate (2) at (1,3); \coordinate (8) at (3,3); 


\draw (1)--(2); \draw (1)--(3); \draw (1)--(4); \draw (1)--(6); \draw (1)--(7);
\draw (1)--(8); \draw (2)--(3); \draw (2)--(4); \draw (2)--(5); \draw (2)--(7);
\draw (2)--(8); \draw (3)--(4); \draw (3)--(5); \draw (3)--(6); \draw (3)--(8);
\draw (4)--(5); \draw (4)--(6); \draw (4)--(7);


\draw [line width=0.05cm, fill=white] (1) circle [radius=0.15] node[below right] {\Large 7}; 
\draw [line width=0.05cm, fill=white] (2) circle [radius=0.15] node[above right] {\Large 8};
\draw [line width=0.05cm, fill=white] (3) circle [radius=0.15] node[above right] {\Large 5};
\draw [line width=0.05cm, fill=white] (4) circle [radius=0.15] node[above left] {\Large 6};
\draw [line width=0.05cm, fill=white] (5) circle [radius=0.15] node[] at (0,6.5) {\Large 3};
\draw [line width=0.05cm, fill=white] (6) circle [radius=0.15] node[below left] {\Large 4};
\draw [line width=0.05cm, fill=white] (7) circle [radius=0.15] node[] at (0,-0.5) {\Large 1};
\draw [line width=0.05cm, fill=white] (8) circle [radius=0.15] node[below right] {\Large 2};


\end{tikzpicture}
}}
\caption{The graph $G_{1,1,1,1}$}
\label{G1111}
\end{minipage}
\end{figure}

We can see that $G_{1,1,1}$ has maximal stable sets $\{1,2,3\}$, $\{1,4\}$, $\{2,5\}$ and $\{3,6\}$.
Thus, we have 
$$\kk[\Stab_{G_{1,1,1}}]=\kk[t_0,t_1t_0,t_2t_0,\ldots,t_6t_0,t_1t_2t_0,t_1t_3t_0,t_2t_3t_0,t_1t_2t_3t_0,t_1t_4t_0,t_2t_5t_0,t_3t_6t_0].$$
\end{ex}

The graph $G_{r_1,\ldots,r_n}$ has the following properties:
\begin{prop}\label{G}
\begin{itemize}
\item[(i)] The maximal cliques of $G_{r_1,\ldots,r_n}$ are precisely $Q_0,Q_1,\ldots,Q_n$.
\item[(ii)] A subset $S\subset V(G_{r_1,\ldots,r_n})$ is a maximal stable set of $G_{r_1,\ldots,r_n}$ if and only if $S=\{v_i,v'_i\} \text{ or } \{v_1,\ldots,v_n\}$ for some $i\in [n]$, $v_i\in Q_i^+$ and $v'_i \in Q_i^-$.
\item[(iii)] The graph $G_{r_1,\ldots,r_n}$ is chordal (and hence perfect), but is not a comparability graph.
\item[(iv)] The stable set ring $\kk[\Stab_{G_{r_1,\ldots,r_n}}]$ is Gorenstein, and $\Cl(\kk[\Stab_{G_{r_1,\ldots,r_n}}])\cong\ZZ^n$.
\item[(v)] One has 
$$\calC(G_{r_1,\ldots,r_n})=\{(z_1,\cdots,z_{n}) \in \RR^n : -r_i\le z_i \leq r_i \text{ for $i\in [n]$} \}.$$
\end{itemize}
\end{prop}

\begin{proof}
(i) By the definition of $G_{r_1,\ldots,r_n}$, we see that $Q_0,Q_1,\ldots,Q_n$ are cliques of $G_{r_1,\ldots,r_n}$ and we can easily check that these are maximal. 
If there exists a maximal clique $Q$ of $G_{r_1,\ldots,r_n}$ which is different from $Q_0,Q_1,\ldots,Q_n$, then there is an element $u_i\in Q\setminus Q_i$ for each $i\in \{0,1,\ldots,n\}$.
We may assume that $u_0\in Q_1$.
Then, since $u_0\in Q_1^+$ and $u_1\notin Q_1$, we have $\{u_0,u_1\}\notin E(G_{r_1,\ldots,r_n})$, a contradiction to $u_0,u_1\in Q$.

\medskip

(ii) It follows from the definition of $G_{r_1,\ldots,r_n}$ that $\{v_i,v'_i\}$ and $\{v_1,\ldots,v_n\}$ are maximal stable sets of $G_{r_1,\ldots,r_n}$.
Suppose that $S$ is a maximal stable set of $G_{r_1,\ldots,r_n}$.
If there exists a vertex $v'_0 \in S\cap Q_0$, then $i\in [n]$ with $v'_0\notin Q_i$ is uniquely determined and we can see that $S=\{v'_0,v_i\}$ for some $v_i \in Q_i^+$ since $\{v'_0,v\} \in E(G_{r_1,\ldots,r_n})$ for any $v \in V(G_{r_1,\ldots,r_n})\setminus Q_i^+$.
If $S\cap Q_0=\emptyset$, then we have $S=\{v_1,\ldots,v_n\}$ for some $v_i\in Q_i^+$ ($i\in [n]$) since $S\subset V(G_{r_1,\ldots,r_n})\setminus Q_0=Q_1^+\cup \cdots \cup Q_n^+$.

\medskip

(iii) We can see that $G_{r_1,\ldots,r_n}$ arises from $n+1$ complete graphs with $d$ vertices by pasting them.
Thus, by Proposition~\ref{past}, $G_{r_1,\ldots,r_n}$ is chordal.

If $G_{r_1,\ldots,r_n}$ is a comparability graph, then so is any induced subgraph of $G_{r_1,\ldots,r_n}$. 
However, for any $n\ge 3$ and $r_1,\ldots,r_n$, the graph $G_{r_1,\ldots,r_n}$ contains $G_{1,1,1}$ as an induced subgraph and we can check that $G_{1,1,1}$ is not a comparability graph, a contradiction.

\medskip

(iv) Since $G_{r_1,\ldots,r_n}$ is perfect and maximal cliques $Q_0,Q_1,\ldots,Q_n$ of $G_{r_1,\ldots,r_n}$ have the same cardinality $d$, the stable set ring $\kk[\Stab_{G_{r_1,\ldots,r_n}}]$ is Gorenstein. Moreover, we have $\Cl(\kk[\Stab_{G_{r_1,\ldots,r_n}}])\cong\ZZ^n$ because $G_{r_1,\ldots,r_n}$ has $n+1$ maximal cliques.

\medskip

(v) From (\ref{weight_stab}), we can see that for $j\in \{0,1,\ldots,2d+n\}$,
\begin{align}\label{Gweight}
\beta_i=\begin{cases}
\eb_j \; & \text{ if  $i=j$ or $i-n \in Q_j^+$}; \\
-\eb_j \; & \text{ if $i-n \in Q_j^-$}.
\end{cases}
\end{align}
Therefore, $\{\bar{\beta}_1,\ldots,\bar{\beta}_{n'}\}=\{\eb_0, \pm\eb_1,\ldots,\pm\eb_n\}$ with $n'=2n+1$. 
Let $\nb \in \ZZ^n$. If there are $i_1,\ldots,i_{n-1}\subset [2n+1]$ such that $\bar{\beta}_{i_k}$'s are linearly independent and $\langle \nb, \bar{\beta}_{i_k} \rangle =0$ for all $k\in [n-1]$, then $\nb$ must be the following form:
$$\nb=m\eb_i \quad \text{ or } \quad \nb=m(\eb_i-\eb_j) \quad (m\in \ZZ\setminus \{0\})$$
for some $i, j\in [n]$. 
By Lemma~\ref{lem} (the observation mentioned in Remark~\ref{rem_sta}), we have
\begin{equation*}
\begin{split}
\calC'(G_{r_1,\ldots,r_n})&=\calW'(\kk[\Stab_{G_{r_1,\ldots,r_n}}]) \\
&\begin{split}
\; =\Big\{(z_1,\cdots,z_{n}) \in \RR^n : -1-|X_{\{i\}\{j\}}^-|\le z_i-z_j &\leq 1+|X_{\{i\}\{j\}}^+| \\
&\text{ for $i,j\in \{0,1,\ldots,n\}$}\Big\}.
\end{split}
\end{split}
\end{equation*}
Moreover, we can determine $X_{\{i\}\{j\}}^{\pm}$ defined in (\ref{XIJ}) as follows: for $i,j\in [n]$,
$$X_{\{i\}\{0\}}^+=Q_i^+, \quad X_{\{i\}\{0\}}^-=Q_i^-, \quad X_{\{i\}\{j\}}^+=Q_i^+\cup Q_j^- \; \text{ and } \; X_{\{i\}\{j\}}^-=Q_j^+\cup Q_i^-.$$

Hence, we get
\begin{align*}
\calC(G_{r_1,\ldots,r_n})=\{(z_1,\cdots,z_{n}) &\in \RR^n : -r_i\le z_i \leq r_i \text{ for $i\in [n]$}, \\
&   -r_i-r_j\le z_i-z_j \leq r_i+r_j \text{ for $i,j\in [n]$}\}.
\end{align*}
Clearly, the inequality $-r_i-r_j\le z_i-z_j \leq r_i+r_j$ can be omitted, and hence we obtain the desired result.
\end{proof}

\medskip

\subsection{Preliminaries on non-commutative resolutions}
Before proving Theorem~\ref{main3}, we recall the definition of NCCRs and prepare some notation and lemmas.

\begin{defi}
Let $R$ be a CM normal domain, let $M \neq 0$ be a reflexive $R$-module, and let $E=\End_R(M)$.
Moreover, let $\gldim E$ denote the global dimension of $E$.
\begin{itemize}
\item We call $E$ a {\em non-commutative resolution} ({\em NCR}, for short) of $R$ if $\gldim E<\infty$. 
\end{itemize}
In addition, suppose that $R$ is Gorenstein.
\begin{itemize}
\item We call $E$ a {\em non-commutative crepant resolution} ({\em NCCR}, for short) of $R$ 
if $E$ is an NCR and is an MCM $R$-module.
\item Moreover, we say that an NCCR $E$ is \textit{splitting} if $M$ is a finite direct sum of rank one reflexive $R$-modules. A splitting NCCR is also called ``toric NCCR'' when $R$ is a toric ring (see \cite{Boc}). 
\end{itemize}
\end{defi}
\noindent Since conic divisorial ideals of a toric ring $R$ are rank one reflexive MCM $R$-modules, the endomorphism ring $E$ of the finite direct sum of some of them is a toric NCCR if $E$ is an MCM $R$-module and $\gldim E<\infty$.

\medskip

We use the same notation as in Section~\ref{pre_toric}.
Let $\calA=\mod(G,S)$ be the category of finitely generated $G$-equivariant $S$-modules. Given $\chi \in X(G)$, let 
$P_\chi=V_\chi \otimes_\kk S$.
Note that $P_\chi \in \calA$ and $M_\chi = P_\chi^G$.
For a subset $\calL \subset X(G)$, we set
$$\displaystyle P_\calL=\bigoplus_{\chi \in \calL}P_\chi \;\text{ and }\; \Lambda_\calL=\End_\calA(P_\calL).$$
Moreover, for $\chi \in X(G)$, let $P_{\calL,\chi}=\Hom_\calA(P_\calL,P_\chi)$. 


Let $Y(G)$ denote the group of one-parameter subgroups of $G$ and let $Y(G)_\RR=Y(G) \otimes_\ZZ \RR$. 
Note that $Y(G) \cong \ZZ^r$ and $Y(G)_\RR \cong \RR^r$. 
We say that $\chi \in X(G)$ is {\em separated from $\calL$ by $\lambda \in Y(G)_\RR$} 
if it holds that $\langle \lambda,\chi \rangle < \langle \lambda,\chi' \rangle$ for each $\chi' \in \calL$. 

\medskip

Our goal is to choose $\calL\subset \calW(R)\cap X(G)$ such that $E=\End_R(M_\calL)$ becomes an NCCR, where $R=S^G$ and $M_{\calL}=\bigoplus_{\chi \in \calL}M_\chi$.
To show $\gldim E<\infty$, we use the following facts:

\begin{itemize}
\setlength{\parskip}{0pt} 
\setlength{\itemsep}{5pt}
\item We see that $\Hom_{\calA}(P_{\chi},P_{\chi'})\cong \Hom_R(M_\chi,M_{\chi'})$ for any $\chi,\chi'\in X(G)$ (\cite[Lemma~3.3]{SpVdB}), which shows that $\gldim \Lambda_\calL < \infty$ implies $\gldim E <\infty$.

\item If $\pdim_{\Lambda_\calL}P_{\calL,\chi}<\infty$ for all $\chi \in X(G)$, where $\pdim_{\Lambda_\calL} P_{\calL,\chi}$ denotes the projective dimension of $P_{\calL,\chi}$ over $\Lambda_\calL$, then $\gldim \Lambda_\calL<\infty$ (\cite[Lemma 10.1]{SpVdB}).

\item  By the same argument as in \cite[Section 10.3]{SpVdB}, we can see that if $\pdim_{\Lambda_\calL}P_{\calL,\chi}<\infty$ for each $\chi \in \widetilde{\calL}$, where $\calL \subset \calW(R)\cap X(G)\subset \widetilde{\calL}\subset X(G)$, 
then $\pdim_{\Lambda_\calL}P_{\calL,\chi}<\infty$ for all $\chi \in X(G)$.

\end{itemize}

\noindent Moreover, we note the following fact to show that $E$ is an MCM $R$-module:
\begin{itemize}
\item Since $\End_R(M_{\calL})\cong \bigoplus_{\chi,\chi'\in \calL}M_{\chi-\chi'}$, the endomorphism ring $E$ is an MCM $R$-module if $M_{\chi-\chi'}$ is an MCM $R$-module for any $\chi,\chi' \in \widetilde{\calL}$. 
\end{itemize}
By summarizing those facts, we obtain the following lemma:

\begin{lem}\label{lem_NCCR}
Let the notation be the same as above.
Then, $E$ is an NCCR of $R$ if there exists $\widetilde{\calL}\subset X(G)$ satisfying the following two conditions:
\begin{itemize}
\item[(a)] $\chi-\chi' \in \widetilde{\calL}$ for any $\chi,\chi'\in \calL$, and $M_\chi$ is an MCM $R$-module for any $\chi \in \widetilde{\calL}$.
\item[(b)] $\calW(R)\cap X(G)\subset \widetilde{\calL}$ and $\pdim_{\Lambda_\calL}P_{\calL,\chi}<\infty$ for each $\chi \in \widetilde{\calL}$.
\end{itemize}
\end{lem}

\noindent In addition, we give another lemma to verify the condition (b) in Lemma~\ref{lem_NCCR}.

\begin{lem}\label{lem:finite2}
Let the notation be the same as above.
\begin{itemize}
\item[(i)] $($\cite[Section 10.1]{SpVdB}$)$ If $\chi$ is in $\calL$, then $P_{\calL, \chi}$ is a right projective $\Lambda_\calL$-module, and hence $\pdim_{\Lambda_\calL}P_{\calL,\chi}<\infty$.

\item[(ii)] $($\cite[Lemma 10.2]{SpVdB}$)$ Let $\chi \in X(G)$ be separated from $\calL$ by $\lambda \in Y(G)_\RR$. Then, we obtain the acyclic complex 
\begin{align*}
0 \rightarrow \bigoplus_{\mu_{d_\lambda}}P_{\calL,\mu_{d_\lambda}} \rightarrow \cdots \rightarrow \bigoplus_{\mu_1}P_{\calL,\mu_1} \rightarrow P_{\calL,\chi} \rightarrow 0, 
\end{align*}
where for each $p \in [d_\lambda]$ with $d_\lambda=|\{ i \in [m] : \langle \beta_i,\lambda \rangle >0\}|$, we let 
$\mu_p=\chi+\beta_{i_1}+\cdots+\beta_{i_p}$ with $\{i_1,\ldots,i_p\} \subset [m]$ and $\langle \beta_{i_j}, \lambda \rangle >0$ for all $j\in [p]$. \\
This implies that $\pdim_{\Lambda_\calL}P_{\calL,\chi}<\infty$ if $\pdim_{\Lambda_\calL}P_{\calL,\mu_p}<\infty$ for each $\mu_p$.
\end{itemize}
\end{lem}

\medskip

\subsection{Proof of Theorem~\ref{main3}}

This subsection is devoted to giving a proof of Theorem~\ref{main3}.

We recall
$$\calL=\{(z_1,\cdots,z_{n}) \in \ZZ^n : 0\le z_i \leq r_i \text{ for $i\in [n]$} \}$$
and set 
$$
\widetilde{\calL}=\calC(G_{r_1,\ldots,r_n})\cap \ZZ^n. 
$$
Note that $\widetilde{\calL}=\{(z_1,\cdots,z_{n}) \in \ZZ^n : -r_i\le z_i \leq r_i \text{ for $i\in [n]$} \}$ from Proposition~\ref{G} (v).

\medskip

In \cite[Theorem 3.6]{HN}, a toric NCCR is given for the Segre product of polynomial rings by taking $\calL,\widetilde{\calL}\subset X(G)$ (these are different from $\calL$ and $\widetilde{\calL}$ defined above, but very similar) and using the same arguments as in Lemma~\ref{lem:finite2}.
Indeed, the proof of Theorem~\ref{main3} can be obtained by the same procedure as that of \cite[Theorem 3.6]{HN}.
However, we give a self-contained proof since $\kk[\Stab_{G_{r_1,\ldots,r_n}}]$ is not the Segre product of polynomial rings.

\begin{proof}[Proof of Theorem~\ref{main3}]
It is enough to show that $\widetilde{\calL}$ satisfies the conditions (a) and (b) in Lemma~\ref{lem_NCCR}.
First, we can check (a) since
$$
\{\chi-\chi' : \chi,\chi'\in \calL\}=\{(z_1,\cdots,z_{n}) \in \ZZ^n : -r_i\le z_i \leq r_i \text{ for $i\in [n]$} \}=\widetilde{\calL}, 
$$
and $M_{\chi}$ is a conic divisorial ideal of $R$ for any $\chi\in \widetilde{\calL}$.

\medskip


Next, we show (b).
We see that $\calL\subset \calW(R)\cap X(G)=\calC(G_{r_1,\ldots,r_n})\cap X(G)=\widetilde{\calL}$.
We have already computed the weights of $R$ in the proof of Proposition~\ref{G} (v).
We set $\bar{\beta}_i=\eb_i$, $\bar{\beta}_{n+i}=-\eb_i$ for $i\in [n]$ and $\bar{\beta}_{2n+1}=\eb_0$.
Note that for $i\in [n]$, the multiplicity $m_i$ of $\bar{\beta}_i$ is equal to $r_i+1$.
We also set 
$$
\widetilde{\calL}_j=\{(z_1,\cdots,z_n) \in \widetilde{\calL} : z_j \geq 0, \cdots, z_n \geq 0\} 
$$
for $j\in [n]$ and   
$$
\widetilde{\calL}_j(k)=\{(z_1,\cdots,z_n) \in \widetilde{\calL}_j : z_j \geq -k \}.
$$
for $0 \leq k \leq r_j$.
Note that $\calL=\widetilde{\calL}_1 \subset \widetilde{\calL}_2 \subset \cdots \subset \widetilde{\calL}_n \subset \widetilde{\calL}$,
$\widetilde{\calL}_{j+1}(0)=\widetilde{\calL}_j(r_j)$ for $j\in [n-1]$ and $\widetilde{\calL}=\widetilde{\calL}_n(r_n)$.
Moreover, for any $j\in [n]$, $k\in [r_j]$ and any $\chi \in \widetilde{\calL}_j(k) \setminus \widetilde{\calL}_{j}(k-1)$, we can see that 
$$\langle \eb_j, \chi \rangle < \langle \eb_j, \chi^\prime\rangle \;\; \text{for any }\chi'\in \widetilde{\calL}_j (k-1).$$ 
Hence, $\chi$ is separated from $\widetilde{\calL}_j(k-1)$ by $\eb_j$, 
and we have that $\langle \eb_j, \bar{\beta}_j  \rangle > 0$ and $\langle \eb_j, \bar{\beta}_{i} \rangle \leq 0$ for any $i\not=j$. 

\medskip

We prove that $\pdim_{\Lambda_\calL}P_{\calL,\chi}<\infty$ for any $j\in [n]$, $k\in \{0,1,\ldots,r_j\}$ and $\chi \in \widetilde{\calL}_j(k)$ by the induction on $j$ and $k$.

\noindent ($j=1$ and $k=0$)
In this case, we have $\widetilde{\calL}_1(0)=\calL$.
Hence, $\pdim_{\Lambda_\calL}P_{\calL,\chi}<\infty$ for any $\chi \in \calL$ by Lemma~\ref{lem:finite2} (i).

\noindent ($j=1$ and $k>0$)
Assume that $\pdim_{\Lambda_\calL}P_{\calL,\chi}<\infty$ for any $\chi \in \widetilde{\calL}_1(k-1)$. 
Then, for any $\chi' \in \widetilde{\calL}_1(k) \setminus \widetilde{\calL}_1(k-1)$, we see that 
$\chi^\prime + \beta_{i_1} + \cdots + \beta_{i_p} \in \widetilde{\calL}_1(k-1)$, where $\beta_{i_1}=\cdots=\beta_{i_p}=\bar{\beta}_1$ and $p\in [r_1+1]$. 
Hence, $\pdim_{\Lambda_\calL}P_{\calL,\chi^\prime}<\infty$ by Lemma~\ref{lem:finite2} (ii). 

\noindent ($j>1$ and $k\ge 0$) 
Assume that $\pdim_{\Lambda_\calL}P_{\calL,\chi}<\infty$ for any $\chi \in \widetilde{\calL}_{j-1}(r_{j-1})$.
The case $k=0$ is trivial since $\widetilde{\calL}_j(0)=\widetilde{\calL}_{j-1}(r_{j-1})$.
Suppose that $k>0$ and $\pdim_{\Lambda_\calL}P_{\calL,\chi}<\infty$ for any $\chi \in \widetilde{\calL}_j(k-1)$. 
Then, for any $\chi' \in \widetilde{\calL}_j(k) \setminus \widetilde{\calL}_j(k-1)$, we see that 
$\chi^\prime + \beta_{i_1} + \cdots + \beta_{i_p} \in \widetilde{\calL}_j(k-1)$, where $\beta_{i_1}=\cdots=\beta_{i_p}=\bar{\beta}_j$ and $p\in [r_j+1]$.
Hence, $\pdim_{\Lambda_\calL}P_{\calL,\chi^\prime}<\infty$ by Lemma~\ref{lem:finite2} (ii).

\medskip

Consequently, we obtain that $\pdim_{\Lambda_\calL}P_{\calL,\chi}<\infty$ for any $\chi \in \widetilde{\calL}_n(r_n)=\widetilde{\calL}$. 

\end{proof}

\bigskip


\end{document}